\documentclass[reqno]{amsart}
\usepackage{amssymb,graphicx}
\usepackage{mathrsfs}
\usepackage{enumerate}
\usepackage{url}
\usepackage{cite}
\usepackage{comment}
\usepackage{graphicx}
\usepackage{tikz}
\usetikzlibrary{arrows,shapes,trees,backgrounds}
\usepackage{amsfonts, amsmath, amssymb, amscd, amsthm, bm, cancel}
\usepackage{setspace}
\usepackage[fontsize=11pt]{fontsize}

\usepackage{geometry}
\allowdisplaybreaks[4]
\theoremstyle{plain}
\newtheorem{thm}{Theorem}[section]
\newtheorem{lem}{Lemma}[section]

\theoremstyle{remark}

\numberwithin{equation}{section}

\def\M{\mathbf{M^{n}}}

\makeatletter %使\section中的内容左对齐
\renewcommand{\section}{\@startsection{section}{1}{0mm}
	{-\baselineskip}{0.5\baselineskip}{\bf\leftline}}
\makeatother
\makeatletter
\@namedef{subjclassname@2020}{\textup{2020} Mathematics Subject Classification}
\makeatother

\begin{document}

\title[Differential Harnack inequalities]
{Differential Harnack inequalities for semilinear\\ parabolic equations on  Riemannian manifolds I:\\ Bakry-\'{E}mery curvature bounded below} 

\author[Zhihao Lu]{Zhihao Lu}
\address[Zhihao Lu]{School of Mathematical Sciences, University of Science and Technology of China, Hefei 230026,P.R.China}
%\email{\href{mailto: Zhihao Lu <lzh139@mail.ustc.edu.cn>}{lzh139@mail.ustc.edu.cn}}

\begin{abstract}
In this paper, we present a unified method for deriving differential Harnack inequalities for positive solutions of the semilinear parabolic equation
\begin{equation*}
	\partial_t u=\Delta_V u+H(u)
\end{equation*}
on complete Riemannian manifolds with Bakry-Émery curvature bounded below. This method transforms the problem of deriving differential Harnack inequalities into solving a related ODE system. As an application of this method, we obtain new and improved estimates for logarithmic-type equations and Yamabe-type equations. Moreover, under the non-negative Bakry-Émery curvature condition, we obtain complete sharp estimates for these equations. As a natural consequence of these results, we also establish sharp Harnack inequalities and Liouville-type theorems for these equations.
\end{abstract}



\thanks{School of Mathematical Sciences,
University of Science and Technology of China, Hefei 230026, P. R. China}

%\thanks{This paper was typeset using \AmS-\LaTeX}

\maketitle
%\tableofcontents

\section{\textbf{Introduction}}

%Let the cylinder
%\begin{eqnarray}
%Q_{r,t}(x_0):=B_{x_0}(r)\times\left[0,t\right].
 %\end{eqnarray}
%where $B_{x_0}(r)$ stands for the geodesic ball whose radius is r with respect to the center $x_0$ and $t>0$.

%First of all,  we  introduce some research progress related to this paper.

In 1986, P. Li and Yau proved a famous gradient estimate for positive solutions of the parabolic equation
\begin{equation}\label{1.1}
	\partial_t u(x,t)=(\Delta -q(x,t))u(x,t)
\end{equation}
on $(\mathbf{M}^n, g)\times[0,\infty)$. When $q(x,t)=0$, they showed the corresponding conclusion for the heat equation
\begin{equation}\label{1.2}
	u_t=\Delta u
\end{equation}
on a Riemannian manifold with Ricci curvature bounded below.

Suppose that $u$ is an arbitrary positive solution of equation \eqref{1.2} on a complete Riemannian manifold $(\mathbf{M}^n, g)$. If $Ricci(\mathbf{M}^n)\geq -Kg$ with $K\geq 0$, Li and Yau obtained the following global gradient estimate:
\begin{equation}
	\frac{|\nabla u|^2}{u^2}-\alpha\frac{u_t}{u}\leq\frac{n\alpha^2K}{2(\alpha-1)}+\frac{n\alpha^2}{2t}\quad\text{for any}\quad\alpha>1.
\end{equation}
This estimate allowed them to obtain important upper and lower bounds for the heat kernel. This gradient estimate is commonly referred to as the differential Harnack inequality.

In 1989, Davies \cite{D} improved the above estimate to
\begin{equation}\label{1.3}
	\frac{|\nabla u|^2}{u^2}-\alpha\frac{u_t}{u}\leq\frac{n\alpha^2K}{4(\alpha-1)}+\frac{n\alpha^2}{2t}\quad\text{for any}\quad\alpha>1.
\end{equation}

In 1993, Hamilton \cite{H} derived the following gradient estimate under the same condition on a closed manifold:
\begin{equation}\label{1.4}
	\frac{|\nabla u|^2}{u^2}-e^{2Kt}\frac{u_t}{u}\leq\frac{ne^{4Kt}}{2t}.
\end{equation}
To the best of our knowledge, Hamilton was the first to consider the coefficient of $\frac{u_t}{u}$ as a function of time. 

In 1999, Bakry and Qian \cite[Theorem 3]{BQ} obtained the following linear-type estimate:
\begin{equation}\label{1.5}
	\frac{|\nabla u|^2}{u^2}-(1+\frac{2Kt}{3})\frac{u_t}{u}\leq\frac{n}{2t}+\frac{nK}{2}(1+\frac{1}{3}Kt).
\end{equation}

In 2011, J.F.Li and X.J.Xu \cite[Theorem 1.1]{LX} generalized Bakry and Qian's result to the following nonlinear type:
\begin{equation}\label{1.6}
\frac{|\nabla u|^2}{u^2}-\alpha(t)
\frac{u_t}{u}\leq \varphi(t).
\end{equation}
Here, $\alpha(t)=1+\frac{\sinh(Kt)\cosh(Kt)-1}{\sinh^2(Kt)}$, $\varphi(t)=\frac{nK}{2}[\coth(Kt)+1]$.

In addition, the authors of this paper also obtained the Bakry-Qian estimate using a different method (see \cite[Theorem 1.2]{LX}). Following the work of Li and Xu, B. Qian \cite{Q} provided a general estimate for the heat equation and summarized the previous results, excluding Hamilton's estimate. In 2017, Bakry, Bolley, and Gentil \cite{BBG} derived a more general estimate for the heat equation under a curvature-dimension condition, which encompassed all previous results using a unified method. In the subsequent description, we will refer to inequalities \eqref{1.5} and \eqref{1.6} as the linear Li-Xu type and Li-Xu type gradient estimates, respectively.

Another direction in generalizing the original Li-Yau gradient estimate is to consider nonlinear parabolic equations. In 1991, J.Y. Li \cite{JL} studied the following equation
\begin{equation}
	\partial_t u=\Delta u+h(x,t)u^p,
\end{equation}
where $p>0$ and $h\in C^{2,1}(\mathbf{M^n}\times[0,\infty))$. They obtained a Li-Yau gradient estimate for equation (1.8) under certain additional conditions on the positive function $h$. They also derived a Liouville theorem for the elliptic equation associated with (1.8), which relaxed the conditions on $h$ originally established by Gidas and Spruck \cite{GS} for the case $1<p<\frac{n}{n-2}$ and $n\geq 4$.

In 2006, Ma \cite{M} investigated the following important equation arising from the Ricci soliton equation:
\begin{equation}
	\partial_t u=\Delta u+au\ln u+bu,
\end{equation}
where $a,b\in\mathbb{R}$ and $a\neq 0$. They obtained Li-Yau type gradient estimates for the case $a<0$. Later, Yang \cite{Y1} derived Li-Yau gradient estimates for the case $a\neq 0$. In 2013, Cao, Ljungberg, and Liu \cite{CLL} improved the Li-Yau estimate on a complete manifold to a sharp estimate for equation (1.9) with $a>0$ under the assumption of non-negative Ricci curvature. They also used these inequalities to derive bounds for the elliptic case of equation (1.9), which were first studied by Chung and Yau \cite{CY}.

Another direction in generalizing the Li-Yau estimate is to consider possible generalizations of Ricci curvature. The Bakry-Émery Ricci curvature, defined as
\begin{equation}
	Ric_V = Ric - \frac{1}{2} L_V g,
\end{equation}
\begin{equation}
	Ric_V^m = Ric_V - \frac{1}{m-n} V^{\flat} \otimes V^{\flat},
\end{equation}
is considered to be a suitable generalization. Here, $m>n$, $V^{\flat}$ is the dual $1$-form of a smooth vector field $V$, and $L_V g$ is the Lie derivative of the metric with respect to $V$. The $\infty$-dimensional and $m$-dimensional Bakry-Émery Ricci curvatures are denoted as $Ric_V$ and $Ric_V^m$, respectively. When $V=-\nabla f$ for some smooth function $f$, the Bakry-Émery Ricci curvature is denoted as $Ric_f$ and $Ric_f^m$.

In 2005, X.D. Li \cite{XL} obtained the Li-Yau estimate under the condition $Ric_f^m \geq -Kg$ on smooth metric measure spaces (or weighted Riemannian manifolds). Hence, his estimate can be considered as a generalization of the Li-Yau estimate. In 2014, Y. Li \cite{YL} further generalized \cite{XL} to the condition $Ric_V^m \geq -K$ using the same technique as in \cite{XL}. It is also worth mentioning that Munteanu and Wang \cite{MW} obtained the Cheng-Yau estimate for $f$-harmonic functions under the condition $Ric_f \geq 0$, with a sharp condition on $f$.

In this paper, we consider a positive function $u(x,t)\in C^{2,1}(\mathbf{M^{n}}\times [0,\infty))$ that solves the nonlinear parabolic equation
\begin{equation}\label{oeq}
	\partial_t u=\Delta_V u+H(u),
\end{equation}
where $\Delta_{V}u:=\Delta u+\left\langle V,\nabla u\right\rangle$, $H\in C^2(0,\infty)$, and $V$ is a $C^1$ vector field. We assume that $(\mathbf{M^{n}},g)$ is an $n$-dimensional complete Riemannian manifold with bounded below Bakry-Émery Ricci curvature. If $\mathbf{M^n}$ is compact, we also allow it to have a convex boundary.

Our main goal in this paper is to establish a unified method for deriving differential Harnack inequalities for solutions of equation (3.1). We provide new results that strengthen and generalize previous work, especially under Ricci non-negative conditions. The differential Harnack inequalities we obtain are sharp for certain specific equations. It should be noted that our method is also applicable to more general equations of the form (3.1) under additional conditions on the first and second partial derivatives of $H(x,t,u)$.

The paper is organized as follows. In Section 2, we state our main results. In Section 3, we prove some basic lemmas and differential Harnack inequalities on compact manifolds. In Section 4, we extend these results to complete Riemannian manifolds. In Sections 5 and 6, under the assumption of non-negative Bakry-Émery Ricci curvature, we derive sharp differential Harnack inequalities, Harnack inequalities, and Liouville-type theorems for logarithmic-type and Yamabe-type equations, respectively. Finally, in the appendix, we provide solutions to $A_3$-systems (see Definition 2.2) for specific equations.

Throughout the paper, universal constants that depend only on the dimension of the manifold will be denoted by $C$ (possibly different in different instances).

\section{\textbf{Main results}}
\subsection{\textbf{Definition of related systems}}
In order to provide a brief overview, we will define some ODE systems that will be used in our analysis. Readers can skip these definitions initially and refer back to them when necessary in the subsequent sections.\\
\textbf{Definition 2.1} Let $\mathbf{M^{n}}$ be a closed or compact manifold with convex boundary. Let $m\in(0,\infty)$, $K\in [0,\infty)$, and $f(x,t)$ be a smooth function on $\mathbf{M^{n}}\times(0,\infty)$.
We call a system an $A_1$-system if the following conditions hold on $\mathbf{M^{n}}\times(0,\infty)$:
\begin{align}
	\begin{cases}
		\frac{4\gamma}{m}c+(\alpha-\gamma)h'(f)+\alpha h''(f)-2K\gamma-\gamma'\ge \frac{\gamma}{\alpha}\left(\frac{4\gamma}{m}c-\alpha'\right)\nonumber\\
		\varphi'-\frac{2\gamma}{m}c^2+\frac{\varphi}{\alpha}\Big(\frac{4\gamma}{m}c-\alpha h'(f)-\alpha'\Big)\ge 0\nonumber\\
		\frac{4\gamma}{m}c-\alpha h'(f)-\alpha'>0\quad\text{and}\quad \alpha,\gamma>0\nonumber
	\end{cases}
\end{align}
Here, $\gamma$, $\alpha$, $\varphi$, and $c$ are $C^1$ functions of time $t$, defined on $(0,\infty)$. $h\in C^2(I)$ and $I$ is an open interval containing the image of $f$.

If an $A_1$-system also satisfies the following boundary condition:
\begin{align}
	\begin{cases}
		\lim\limits_{t\to0^{+}}\alpha \qquad\text{exists}\\
		\lim\limits_{t\to0^{+}}\gamma \qquad\text{exists}\\
		\lim\limits_{t\to0^{+}}\varphi=\infty,\nonumber
	\end{cases}
\end{align}
we call it an $A_2$-system.

More precisely, we call the functions $\gamma$, $\alpha$, $\varphi$, $c$, $h$, and $f$ satisfy the $A_1$-system (or $A_2$-system).

\noindent\textbf{Definition 2.2} Let $(\mathbf{M^{n}},g)$ be a complete Riemannian manifold. Let $m\in(0,\infty)$, $K\in [0,\infty)$ and $f(x,t)$ be a smooth function on $\mathbf{M^{n}}\times(0,\infty)$.
 We call a system an $A_3$-system if 
\begin{align}
	\begin{cases}
			\frac{4\gamma}{m}c+(\alpha-\gamma)h'(f)+\alpha h''(f)-2K\gamma-\gamma'\ge \frac{\gamma}{\alpha}\Big(\frac{4\gamma}{m}c-\alpha'\Big)\nonumber\\	
		\varphi'-\frac{2\gamma}{m}c^2+\frac{\varphi}{\alpha}\Big(\frac{4\gamma}{m}c-\alpha h'(f)-\alpha'\Big)\ge 0\nonumber\\
		h'(f)+\frac{\alpha'}{\alpha}+\frac{\beta'}{\beta}-\frac{4\gamma\varphi}{m\alpha^2}\le 0\nonumber\\
		\beta(0)=0\quad\text{and}\quad\beta >0\nonumber
	\end{cases}
\end{align}
 on $\mathbf{M^{n}}\times(0,\infty)$ and any one of the following three conditions 
\begin{equation}
	\min\{\alpha-\gamma,\gamma\}\ge\epsilon>0,\quad \text{and} \quad \alpha,\beta \quad \text{are non-decreasing}\qquad(I)\nonumber
\end{equation}
\begin{equation}
\alpha>\gamma>0, \quad\text{and}\quad\alpha,\frac{\alpha^2\beta}{\gamma} \quad\text{are non-decreasing,}\quad	\frac{\beta}{\gamma^2(\alpha-\gamma)} \quad\text{is bounded on}\quad (0,\infty)\quad(II)\nonumber
\end{equation}
\begin{equation}
\alpha>\gamma>0, \quad\text{and}\quad\alpha,\frac{\alpha^2}{\gamma},\beta \quad\text{are non-decreasing,}\quad	\frac{\beta}{\alpha-\gamma} \quad\text{is bounded on}\quad (0,\infty)\qquad(III)\nonumber
\end{equation}
holds. Here, $\gamma$, $\alpha$, $\varphi$, and $c$ are $C^1$ functions of time defined on $(0,\infty)$. $\beta\in C^1([0,\infty))$, $h$ is a $C^2$ function on an interval $I$ that contains the image of $f$. 
More precisely, we also call the functions $\gamma$, $\alpha$, $\varphi$, $\beta$, $h$, $c$, and $f$ satisfy the $A_3$-system.

\noindent\textbf{Remark 2.1}\\
(i) The $A_2$-system mainly deals with the differential Harnack inequalities on closed manifolds or compact manifolds with convex boundary. The $A_3$-system mainly deals with the differential Harnack inequalities on complete noncompact manifolds.\\
(ii) We can also define $A_1$-$A_3$ systems on a finite time interval and a finite geodesic ball $B(x_0,R)$ (an open geodesic ball with center $x_0$ and radius $R$) for general consideration. In this case, we just need to change $\mathbf{M^{n}}$ and $(0,\infty)$ to $B(x_0,R)$ and $(0,T]$ in Definitions 1.1 and 1.2 for some $R<\infty$ and $T<\infty$.
We also call the corresponding system a local $A_3$-system. To be precise, we also call the functions $\gamma$, $\alpha$, $\varphi$, $\beta$, $h$, $c$, and $f$ satisfy the local $A_3$-system on $B(x_0,R)\times(0,T]$.

\subsection{\textbf{General differential Harnack inequalities}}

\begin{thm}\label{CGD}
Let $(\mathbf{M^{n}},g)$ be an $n$-dimensional closed Riemannian manifold with $Ric_V^m \ge-Kg$ for $m>n$ and $K\ge 0$. Let $u(x,t)$ be a smooth positive solution of \eqref{oeq} on $\mathbf{M^{n}}\times[0,\infty)$, $f:=\ln u$, and $h(t):=\frac{H(e^t)}{e^t}$. If there are smooth functions $\gamma$, $\alpha$, $\varphi$, $c$ defined on $(0,\infty)$ and $f$, $h$ satisfy the $A_2$-system, then we have the following global differential Harnack inequality:
\begin{equation}\label{1.11}
	\gamma\frac{|\nabla u|^2}{u^2}-\alpha \frac{u_t}{u}+\alpha h(f)-\varphi\leq 0\qquad\text{on}\quad \mathbf{M^{n}}\times(0,\infty).
\end{equation}
If $(\mathbf{M^{n}},g)$ is an $n$-dimensional compact manifold with convex boundary, we consider the positive solution of the Neumann problem:
\begin{align}
	\begin{cases}
		\partial_t u=\Delta_{V}u+H(u)\quad\text{on}\quad \mathbf{M^{n}}\\
		\frac{\partial u}{\partial \textbf{n}}=0 \qquad\qquad\qquad\text{on}\quad\partial\mathbf{M^{n}},
	\end{cases}
\end{align}
where $\textbf{n}$ is the unit outer normal vector field. Then the same estimate \eqref{1.11} still holds.

\end{thm}

\begin{thm}\label{GD}
	Let $(\mathbf{M^{n}},g)$ be an $n$-dimensional complete Riemannian manifold with $Ric_V^m \ge-Kg$ for $m>n$ and $K\ge 0$. Let $u(x,t)$ be a smooth positive solution to the equation
	\eqref{oeq} on $B(x_0,2R)\times[0,\infty)$, $f:=\ln u$ and $h(t):=\frac{H(e^t)}{e^t}$. If there are functions $\gamma$, $\alpha$, $\varphi$, $\beta$, $c$ defined on $(0,\infty)$ and $f$, $h$ satisfy local $A_3$-system on $B(x_0,2R)\times(0,\infty)$, then we have the following  local differential Harnark inequalities:
	\begin{align}
	&\gamma\frac{|\nabla u|^2}{u^2}-\alpha \frac{u_t}{u}+\alpha h(f)-\varphi\nonumber\\
	\leq&
	\begin{cases}
		\frac{m\alpha^2\beta(T)}{2\epsilon}\Big[\frac{C}{R^2}(1+\sqrt{K}R\coth(\sqrt{K}R))+\frac{m\alpha^2(T)}{4\epsilon^2}\cdot\frac{C}{R^2}\Big]\quad\,\,	\text{if}\quad (I)\quad\text{holds}\vspace{2ex}\nonumber\\
		\frac{m\alpha^2\beta}{2\gamma}(T)\cdot\frac{C}{R^2}(1+\sqrt{K}R\coth(\sqrt{K}R))+m^2\alpha^4(T)\frac{C}{R^2}\quad	\text{if}\quad (II)\quad\text{holds}\vspace{2ex}\nonumber\\
		\frac{m\alpha^2\beta}{2\gamma}(T)\frac{C}{R^2}(1+\sqrt{K}R\coth(\sqrt{K}R))+\frac{m^2\alpha^4}{\gamma^2}(T)\cdot\frac{C}{R^2}\quad\,\,	\text{if}\quad (III)\quad\text{holds}\nonumber
	\end{cases}
	\end{align}
on $B(x_0,R)\times(0,\infty)$.

Moreover, if the solution exists on the whole time-space $\mathbf{M^{n}}\times[0,\infty)$ and the functions $\gamma$, $\alpha$, $\varphi$, $\beta$, $c$ are defined on $(0,\infty)$ and $f$, $h$ satisfy the $A_3$-system, we have the following global differential Harnack inequality:
	\begin{equation}\label{gdhk}
	\gamma\frac{|\nabla u|^2}{u^2}-\alpha \frac{u_t}{u}+\alpha h(f)-\varphi\leq 0\qquad\text{on\,\,\, $\M\times(0,\infty)$.}
	\end{equation}
\end{thm}
\noindent\textbf{Remark 2.2}\\
(i) From our proof below (see Section 3 and Section 4), our results are also valid for solutions on $B(x_0,2R)\times[0,T]$ or $\mathbf{M^{n}}\times[0,T]$ (in this case, we just need to find functions which solve the $A_2$ or $A_3$-system on the time interval $(0,T]$). If one considers a positive solution of equation (1.12) on $B(x_0,2R)\times[T_0,T_0+T]$, the estimate in Theorem 2.2 also holds by time translation. Actually, if our solution is defined on $\mathbf{M^{n}}\times(0,\infty)$ (or $\mathbf{M^{n}}\times(0,T]$), we also have the same estimates as (2.1) and (2.3) by a simple time-translation argument.\\
(ii) To our knowledge, there are few results about differential Harnack inequalities under the condition $Ric_V\ge-Kg$ (without any condition on $V$) due to the lack of an important Laplacian comparison theorem. At the same time, the key lemma below (Lemma 3.3) also requires a lower bound on $Ric^m_V$ rather than $Ric_V$.

\subsection{\textbf{Application to logarithmic type equation}}
In this subsection, we provide sharp differential Harnack inequalities, sharp Harnack inequalities, and Liouville type theorems for the logarithmic type equation under the non-negative Bakry-Émery Ricci curvature condition.

The logarithmic type equation arises from the study of Ricci solitons and Log-Sobolev inequalities (see \cite{CLL,CY,G}). It can be derived from the Ricci soliton equation, as shown in \cite{CLL,CLN,M}. If a function on a compact manifold achieves the sharp constant in the Log-Sobolev inequality, it satisfies the corresponding elliptic equation of the form (2.5) with $a>0$ (see \cite{CY}). We rewrite the equation as follows:
\begin{equation}
	\partial_{t}w=\Delta_{V}w+aw\ln w+bw,
\end{equation}
where $a$ and $b$ are two real constants. If $a=0$, then equation (2.4) becomes a linear equation whose differential Harnack inequalities are directly given in Appendix 7.1. If $a\neq 0$, by letting $w=e^{-\frac{b}{a}}u$, one obtains an equivalent equation as follows:
\begin{equation}\label{logeq}
	\partial_{t}u=\Delta_{V}u+au\ln u.
\end{equation}
Thus, we focus on the equation \eqref{logeq} for simplicity and state our basic results about this equation.

\begin{thm}
	Let $(\mathbf{M^{n}},g)$ be an $n$-dimensional complete Riemannian manifold with $Ric_V^m \ge 0$ for $m>n$. Let $u(x,t)$ be a smooth positive solution to the equation
	\eqref{logeq} on $\mathbf{M^{n}}\times(0,\infty)$ with $a\neq 0$ and $f:=\ln u$. For the following cases:\\
	$(i)$ $\mathbf{M^{n}}$ is closed;\\
	$(ii)$ $\mathbf{M^{n}}$ is complete noncompact;\\
	$(iii)$ $\mathbf{M^{n}}$ is compact with convex boundry and $u$ solves Neumann problem (2.2),\\
	we have the following differential Harnack inequality:
	\begin{equation}\label{sldhk}
		\Delta_{V} f +\frac{ma}{2(1-e^{-at})}\ge 0.
	\end{equation}
	At case $(ii)$, \eqref{sldhk} is sharp.	
\end{thm}
\begin{thm}
	Let $(\mathbf{M^{n}},g)$ be an n-dimensional complete Riemannian manifold. Let $u$ be a positive solution of \eqref{logeq} on $\mathbf{M^{n}}\times(0,\infty)$ which possesses  global differential Harnark inequality \eqref{gdhk}. If we set $f:=\ln u$ and suppose that $x_1$, $x_2\in \mathbf{M^{n}}$, $0<t_1<t_2$, then we have the following Harnack inequality:
\begin{equation}
e^{-at_2}f(x_2,t_2)-e^{-at_1}f(x_1,t_1)\ge\int_{t_1}^{t_2}e^{-at}\left(\frac{-\varphi}{\alpha}-\frac{\alpha}{4\gamma}\left|\frac{dl}{dt}\right|^2\right)dt,	
\end{equation}
where $l:[t_1,t_2]\to\mathbf{M^{n}}$ is a smooth path connecting $x_1$ and $x_2$.

Moreover, if $\alpha$,$\gamma$ are two positive constants, we have 
\begin{equation}\label{loghk}
e^{-at_2}f(x_2,t_2)-e^{-at_1}f(x_1,t_1)\ge\int_{t_1}^{t_2}-e^{-at}\cdot \frac{\varphi}{\alpha} dt-\frac{\alpha}{4\gamma}\cdot\frac{a\cdot d(x_1,x_2)^2}{e^{at_2}-e^{at_1}}.	
\end{equation}	
Especially, at case $(ii)$ in Theorem 2.3, the corresponding Harnack inequality of \eqref{sldhk} is also sharp. 	
		
\end{thm}

\begin{thm}
	Let $(\mathbf{M^{n}},g)$ be an $n$-dimensional complete Riemannian manifold with $Ric_V^m \ge 0$ and $m>n$. We write the corresponding elliptic equation of \eqref{logeq} as follows:
	\begin{equation}\label{elog}
	\Delta_V u+au\ln u=0.
	\end{equation}
	Then we have the following Liouville type results:\vspace{2mm}\\
	{\rm(1)} Suppose $\mathbf{M^{n}}$ is a closed manifold. Let $u(x,t)$ be a smooth positive solution of
	\eqref{logeq} on $\mathbf{M^{n}}\times(-\infty,0)$ with $a<0$. Then $u(x,t)=e^{ce^{at}}$ for some $c\in\mathbb{R}$. Moreover, if $u$ is a positive solution of \eqref{elog}, then $u\equiv 1$.\\
	{\rm(2)} Let $u(x,t)$ be a smooth positive solution of
	$(2.5)$ on $\mathbf{M^{n}}\times(-\infty,0)$ with $a<0$. And we define the following partition of $\mathbf{M^{n}}$: 
	\begin{eqnarray}
	S_1=\{x\in\mathbf{M^{n}}:u(x,t)>1\quad \text{for any} \quad t\in(-\infty,0)\},\nonumber\\
	S_2=\{x\in\mathbf{M^{n}}:u(x,t)\equiv 1\quad \text{for any} \quad t\in(-\infty,0)\},\nonumber\\
	S_3=\{x\in\mathbf{M^{n}}:u(x,t)<1\quad \text{for some} \quad t\in(-\infty,0)\}.\nonumber
	\end{eqnarray}
	Then we have\\
	{\rm (2a)} For $x\in S_1$, we have $u(x,\cdot)=e^{O(e^{at})}$ as $t\to-\infty$.\\
	{\rm (2b)} For $x\in S_3$, we have $u(x,\cdot)=e^{O(-e^{at})}$ as $t\to-\infty$.\\
	{\rm (2c)} If $u$ has lower bound $\delta>0$, then it has lower bound $1$ and for any $x\in\mathbf{M^{n}} $, we have $u(x,\cdot)=O(e^{ce^{at}})$ as $t\to-\infty$ for some $c\ge 0$.\\
	{\rm (2d)} If $u(\cdot,t_0)\le 1$ for some $t_0$, then $u(x,t)\le 1$ for any $t\le t_0$. Moreover, if $u\le \delta<1$, then $u(x,t)\le e^{\ln \delta\cdot e^{at}}$.
	
	The growth control on time in {\rm (2a)-(2d)} are sharp.\vspace{1mm}\\
	{\rm (3)} Let $u(x,t)$ be a smooth positive solution of
	\eqref{logeq} on $\mathbf{M^{n}}\times(-\infty,0)$ with $a>0$. Then\\
	{\rm (3a)} If $u(x_0,t_0)<e^{\frac{m}{2}}$, then $u(x_0,t)<e^{\frac{m}{2}}$ for $t\le t_0$.\\
	{\rm (3b)} For any $t_0\in(-\infty,0)$, we have 
	\begin{equation}
	\ln u(x,t)\le\frac{m}{2}+[\ln u(x,t_0)-\frac{m}{2}] \cdot e^{a(t-t_0)}\quad \text{for}\quad t\le t_0.
	\end{equation} 
	
	Moreover, $\limsup\limits_{t\to-\infty}u(x,t)\le e^{\frac{m}{2}}$ for any $x\in\mathbf{M^{n}}$. If $V=\mathbf{0}$ {\rm(i.e. $Ric\ge 0$)}, we can replace $m$ by $n$.\\
	{\rm (4)} Let $u(x,t)$ be a smooth positive solution of
	\eqref{logeq} on $\mathbf{M^{n}}\times(0,\infty)$ with $a>0$. Then we define the following partition of $\mathbf{M^{n}}$:
	\begin{eqnarray}
	Z_1=\{x\in\mathbf{M^{n}}:\limsup\limits_{t\to\infty}u(x,t)> e^{\frac{m}{2}}\},\nonumber\\
	Z_2=\{x\in\mathbf{M^{n}}:\liminf\limits_{t\to\infty}u(x,t)< e^{\frac{m}{2}}\},\nonumber\\
	Z_3=\{x\in\mathbf{M^{n}}:\lim\limits_{t\to\infty}u(x,t)= e^{\frac{m}{2}}\}.\nonumber
	\end{eqnarray}
	{\rm (4a)} For $x\in Z_1$, then there exists $c>0$ and $t_0\in(0,\infty)$ such that
	\begin{equation}
	u(x,t)\ge e^{-\frac{m}{2}e^{at}\ln (1-e^{-at})+ce^{at}}\quad \text{for}\quad t\ge t_0.
	\end{equation} 

	Moreover, if  $\ln u\ge\delta>\frac{m}{2}$, then there exists $c=c(\delta)>0$ and $t_0=t(\delta)$ such that
	\begin{equation}
	u(x,t)\ge e^{-\frac{m}{2}e^{at}\ln (1-e^{-at})+ce^{at}}\quad \text{in}\quad (t_0,\infty).\nonumber
	\end{equation}
	{\rm (4b)} For $x\in Z_2\cup Z_3$, then $u(x,t)\le e^{-\frac{m}{2}e^{at}\ln(1-e^{-at})}$. Therefore  $\limsup\limits_{t\to\infty}u(x,t)\le e^{\frac{m}{2}}$ for $x\in Z_2\cup Z_3$. Then we have $Z_1\cap Z_2=\emptyset$ which yields that $Z_1$,$Z_2$,$Z_3$ indeed form a partition of $\mathbf{M^{n}}$.
	
	If $V=\mathbf{0}$, we can replace $m$ by $n$. In this case, the decay estimates in {\rm (4a)} and {\rm (4b)} are sharp.\\
	{\rm (5)} Let $u(x,t)$ be a smooth positive solution to the equation
	$(2.5)$ on $\mathbf{M^{n}}\times(0,\infty)$. Then for any $t_0\in(0,\infty)$ and $x\in\mathbf{M^{n}}$, we have the following decay estimate on time:
	\begin{equation}
	\ln u(x,t)\ge e^{a(t-t_0)}\ln u(x,t_0)+\frac{m}{2}e^{at}\ln\Big(\frac{e^{-at_0}-1}{e^{-at}-1}\Big)\quad \text{for}\quad t\ge t_0,
	\end{equation}
	i.e.
	\begin{equation}
	F(x,t)\quad\text{is non-decreasing on time}.\nonumber
	\end{equation}
	Here $F(x,t):=e^{-at}\ln u(x,t)+\frac{m}{2}\ln|e^{-at}-1|$.
	
	Moreover, at $a<0$ case, one must have $\liminf\limits_{t\to\infty}u(x,t)\ge 1$ for any $x\in\mathbf{M^{n}}$. If $V=\mathbf{0}$, we can replace $m$ by $n$. In this case, $(2.12)$ is sharp.\\
	{\rm (6)} Let $u(x)$ be a smooth positive solution to the equation
	\eqref{elog} on $\mathbf{M^{n}}$. If $a<0$, then $u\ge 1$. If $a>0$, then $u\le e^{\frac{m}{2}}$. 
\end{thm}

\noindent\textbf{Remark 2.3}\\
(i) If one wants to get a general bound of \eqref{elog} for $Ric_V^m\ge-K$ case, one can use the improved differential Harnack inequalities in Appendix 7.2. \\
(ii) To our knowledge, our results are new under the Bakry-Émery curvature case. Under the Ricci curvature condition, for the case $a>0$, the bound in (2.9) (with $V$ also vanishing) was first derived by \cite{CY}. For the case $a<0$, the bound $e^{-\frac{n}{16}}$ (in \cite{CLL}) was the best known before our present paper. Due to the lack of the sharp differential Harnack inequality (2.6), it was not possible to obtain the sharp bound of 1 for the case $a<0$ for a long time.
\\
(iii) Before our consideration, in \cite{CLL}, the authors derived the sharp differential Harnack inequality for the case $a>0$. However, the proof we provide below is new and simpler in this case.
When $a<0$, the situation is a bit more complex and circuitous.

\subsection{\textbf{Application to Yamabe type equation}}
In this subsection, we provide differential Harnack inequalities, Harnack inequalities, and Liouville type theorems for the Yamabe type equation under the non-negative Bakry-Émery Ricci curvature condition.

The following Yamabe type equation arises from the scalar curvature equation and can be rewritten as follows:
\begin{equation}\label{saeq}
	\partial_{t}u=\Delta_{V}u+au+bu^p,
\end{equation}
where $a$, $b$, and $p$ are real constants. There are some specific cases:
(i) $a>0$, $b=-a$, and $p=2$, which corresponds to the Fisher-KPP equation.
(ii) $a=-b=1$ and $p=3$, which corresponds to the parabolic Allen-Cahn equation.
(iii) $a>0$, $b<0$, and $p=3$, which corresponds to the Newell-Whitehead equation, extending the parabolic Allen-Cahn equation.

For a more general setting, we can consider a generalization of  equation \eqref{saeq} as follows:
\begin{equation}\label{geq}
	\partial_{t}u=\Delta_{V}u+\sum_{i=1}^{N}a_i u^{p_i},
\end{equation}
where $a_i$ and $p_i$ are real numbers, $N$ is a natural number, and $p_1<\cdots<p_N$.
 
\begin{thm}
	Let $(\mathbf{M^{n}},g)$ be an $n$-dimensional complete Riemannian manifold with $Ric_V^m \ge 0$ for $m>n$. Assume $a_i\ge 0$ and $p_i\le 1$ for $i=1,\cdots,N$. Let $u(x,t)$ be a smooth positive solution of
	{\rm(\ref{geq})} on $\mathbf{M^{n}}\times(0,\infty)$ and $f:=\ln u$. For the following cases:\\
	$(i)$ $\mathbf{M^{n}}$ is complete;\\
	$(ii)$ $\mathbf{M^{n}}$ is compact with convex boundry and $u$ solves Neumann problem (2.2),\\
	then we have differential Harnack inequality:
	\begin{equation}\label{sgdhk}
	\Delta_{V} f +\frac{m}{2t}\ge 0.
	\end{equation}	
\end{thm}

\noindent\textbf{Remark 2.4}\\
Under the non-negative Ricci curvature condition, if $a_i=0$, then \eqref{sgdhk} represents Li-Yau's differential Harnack inequality, which is sharp in this case. For any $a$, $b$, $p \in \mathbb{R}$ and $Ric_V^m\ge-K$, one can obtain other differential Harnack inequalities for the equation \eqref{saeq} as presented in Appendix 7.3.

\begin{thm}
	Let $(\mathbf{M^{n}},g)$ be a complete Riemannian manifold. Let $u$ be a positive solution of {\rm(\ref{geq})} on $\mathbf{M^{n}}$ and the global differential Harnark inequality $(2.2)$ holds (all symbols are same as in Theorem 2.2). For simplicity, we suppose $a_i\ge 0$ and $p_N\le 1$. If we set $f:=\ln u$ and suppose that $x_1$, $x_2\in \mathbf{M^{n}}$, $0<t_1<t_2$, then we have \vspace{2mm}\\
	\begin{equation}
	f(x_2,t_2)-f(x_1,t_1)\ge\int_{t_1}^{t_2}\frac{-\varphi}{\alpha}-\frac{\alpha}{4\gamma}\left|\frac{dl}{dt}\right|^2dt.	
	\end{equation}
	Here $l$ is any smooth path connecting $x_1$ and $x_2$.
	
	Moreover, if $\alpha$,$\gamma$ are two positive constants (such as in Li-Yau type inequality) and $l$ is a minimizing geodesic with constant speed, we have 
	\begin{equation}
	f(x_2,t_2)-f(x_1,t_1)\ge\int_{t_1}^{t_2}- \frac{\varphi}{\alpha} dt-\frac{\alpha}{4\gamma}\cdot\frac{ d(x_1,x_2)^2}{t_2-t_1}.	
	\end{equation}	
\end{thm}

\noindent\textbf{Remark 2.5}\\
If we impose a bound condition (or local upper bound) on the positive solution $u$ of equation \eqref{saeq}, it is possible to derive Harnack inequalities (or their local versions) for \eqref{saeq} in other cases by utilizing the differential Harnack inequalities provided in Appendix 7.3. This can be done by employing a similar argument to the proof of Theorem 2.7.

\begin{thm}
	Let $(\mathbf{M^{n}},g)$ be an $n$-dimensional complete Riemannian manifold with $Ric_V^m \ge 0$ for $m>n$. We write the corresponding elliptic equation of {\rm(\ref{geq})} as follows:
	\begin{equation}\label{ey}
	\Delta_V u+\sum_{i=1}^{N}a_i u^{p_i}=0.
	\end{equation}
	Then we have the following Liouville type results:\vspace{1mm}\\
	$(a)$ For $a_i=0$ case, let $u(x,t)$ be a smooth positive solution of the equation {\rm(\ref{geq})}  on $\mathbf{M^{n}}\times(-\infty,0)$, then $u(x,\cdot)$ is non-decreasing for any $x\in\mathbf{M^{n}}$.	Moreover, for some fixed $t_0\in(-\infty,0)$ and some $x_0\in\mathbf{M^{n}}$, if we have the following additional growth condition:
	\begin{equation}
	u(x,t_0)= e^{o(d(x,x_0))}\quad \text{as}\quad d(x,x_0)\to\infty,
	\end{equation}
	then $u$ is constant.\vspace{1mm}\\
	$(b)$ For $a_i=0$ case, let $u(x,t)$ be a smooth positive solution of the equation {\rm(\ref{geq})}  on $\mathbf{M^{n}}\times(0,\infty)$, then $t^{\frac{m}{2}}u(x,\cdot)$ is non-decreasing for any $x\in\mathbf{M^{n}}$.\vspace{1mm}\\
	$(c)$ Suppose $a_i> 0$, $p_1<1$ and $p_N\le 1$. Then there does not exist a positive solution to equation \eqref{geq} on $\mathbf{M^{n}}\times(-\infty,0)$. 
	Moreover, any non-negative solution of {\rm(\ref{geq})} must vanish before some time $T\in(-\infty,0)$.\vspace{1mm}\\
	% If $p_0\in(0,1)$, let $u(x,t)$ be a nonnegative smooth solution of {\rm(\ref{general})} on $\mathbf{M^{n}}\times(-\infty,0)$, then $u\equiv 0$
	 %when $t\le T_0$ for some $T_0\in(-\infty,0)$. \\
	 %If $p_0=0$, there not exists nonnegative smooth solution of $(2.13)$ on $\mathbf{M^{n}}\times(-\infty,0)$.\\.\vspace{2mm}\\
	$(d)$ Suppose $a_i> 0$, $p_1< 1$ and $p_N\le 1$. Then there does not exist a positive solution to equation \eqref{geq} on $\mathbf{M^{n}}\times(0,\infty)$ such that $u(x,t)=o(t^{\frac{1}{1-p_1}})$ for some $x\in\mathbf{M^{n}}$. This growth condition on time is sharp because of space-independent solution $u_0(t)=(a_1(1-p_1)t)^{\frac{1}{1-p_1}}$ (when $N=1$).
\end{thm}

\noindent\textbf{Remark 2.6}\\
(i) Theorem 2.8 (a) is attributed to Lin-Zhang \cite{LZ} under the condition of Ricci curvature. When $K>0$, \cite{MS} also obtained that $u$ is time-independent with a similar growth condition. \\
(ii) When $N=1$, $p_1\in(0,1)$ and $V=0$, Zhu \cite{Z} proved a Liouville theorem (\cite[Theorem 2]{Z}) for ancient solutions of equation \eqref{geq} with an additional growth condition on the solution, using a Souplet-Zhang type estimate (see \cite{SZ}). Here, we are able to remove this condition by our parabolic estimate.\\
(iii) When $N=1$, $p_1<0$, and $V=0$, Yang \cite{Y2} obtained a Liouville theorem for the elliptic equation \eqref{ey} (with $N=1$). Here, we can directly generalize their result using (c) in Theorem 2.8. We can also use (d) to obtain the same result.

\section{\textbf{Preliminary }}
In this section,  we first provide some useful lemmas and basic computation for proving differential Harnack inequalities of the positive solutions to the equation \eqref{oeq}. Then we give a direct proof of Theorem 2.1 by maximal principle.

\subsection{\textbf{Basic Lemmas}}

Before the proof of the main theorems, we need some lemmas.
At first, we consider a more general equation
\begin{equation}\label{goeq}
\partial_t u(x,t)=\Delta_V u(x,t)+H(x,t,u)
\end{equation}
on a complete Riemannian manifold $(\mathbf{M^{n}}, g)$ with $Ric^m_V\ge -Kg$.
Let $u(x,t)$ be  a positive solution of \eqref{goeq} on $\Omega\times(0,\infty)$, where $\Omega$ is an open set of $\mathbf{M^{n}}$.
We set 
 \begin{eqnarray}
 f=\ln u,
 \end{eqnarray}
 \begin{eqnarray}
 G(x,t,f)=H(x,t,e^{f})\cdot e^{-f}.
 \end{eqnarray}
Observe that
 \begin{eqnarray}\label{feq}
f_t&=&\frac{u_t}{u}\nonumber\\
&=&\frac{1}{u}\Big(\Delta_V u+H\Big)\nonumber\\
&=&e^{-f}\Big(\Delta_V (e^f)+H(x,t,e^f)\Big)\nonumber\\
&=&e^{-f}\Big(e^f\Delta_V f+e^f |\nabla f|^2+H(x,t,e^f)\Big)\nonumber\\
&=&\Delta_V f+|\nabla f|^2+G(x,t,f).
\end{eqnarray}
Define the $V$-heat operator
\begin{equation}
\mathscr{L} =\Delta_V-\partial_{t}.\nonumber
\end{equation}
It is easy to see
\begin{equation}\label{lr}
\mathscr{L}(k\cdot l)=\mathscr{L}k\cdot l+\mathscr{L}l\cdot k+2\left\langle\nabla k,\nabla l\right\rangle
\end{equation}
for any $k,l\in C^{2,1}(\mathbf{M^n}\times(0,\infty))$.
\begin{lem}\label{l31}
Let $F(x,t):=\gamma(t)|\nabla f|^{2}-\alpha(t) f_t+\alpha(t) G-\varphi(t)$ on $\Omega\times(0,\infty)$. Then
\begin{eqnarray}\label{36}
\mathscr{L}F&=&2\gamma|\nabla^2 f|^2+2\gamma Ric_V(\nabla f,\nabla f)-\gamma'|\nabla f|^2-2\left\langle \nabla f,\nabla F\right\rangle\nonumber\\
&&+\alpha\Big(\frac{\partial G}{\partial t}+\frac{\partial G}{\partial f}\cdot f_t\Big)+2(\alpha-\gamma)\left\langle \nabla f,\nabla_x G+\frac{\partial G}{\partial f}\nabla f\right\rangle\nonumber\\
&&+\alpha' f_t-\alpha' G+\alpha \mathscr{L}(G)+\varphi',
\end{eqnarray}
where $\gamma,\alpha,\varphi\in C^1(0,\infty)$ are undetermined functions.
\end{lem}

\begin{proof}[\textbf{Proof}]
By using the Bochner-Weitzenb\"{o}ck formula and \eqref{lr}, we have
\begin{eqnarray}\label{37}
\mathscr{L}F&=&\gamma\mathscr{L}(|\nabla f|^2)-\gamma'|\nabla f|^2-\mathscr{L}(\alpha)f_t-\alpha\mathscr{L}(f_t)+\mathscr{L}(\alpha)G+\alpha\mathscr{L}(G)-\mathscr{L}\varphi\nonumber\\
&=& 2\gamma\Big[|\nabla^2 f|^{2}+Ric_{V}(\nabla f, \nabla f)+\left\langle \nabla f,\nabla\Delta_V f\right\rangle-\left\langle \nabla f,\nabla f_t\right\rangle\Big]-\gamma'|\nabla f|^2 \nonumber\\
&&+\alpha' f_t-\alpha\mathscr{L}(f_t)-\alpha'G+\alpha\mathscr{L}(G)+\varphi'\nonumber\\
&=&2\gamma\Big[|\nabla^2 f|^{2}+Ric_V(\nabla f, \nabla f)+\left\langle \nabla f,\nabla \mathscr{L} f\right\rangle\Big]-\gamma'|\nabla f|^2\nonumber\\
&&+\alpha' f_t-\alpha\mathscr{L}(f_t)-\alpha'G+\alpha\mathscr{L}(G)+\varphi'.
\end{eqnarray}
By \eqref{feq}, we obtain
\begin{eqnarray}
\mathscr{L}f&=&-|\nabla f|^2-G,\\
\mathscr{L}(f_t)&=&(\mathscr{L}f)_t\nonumber\\
&=&-2\left\langle \nabla f,\nabla f_t\right\rangle-\partial_tG-\partial_f G\cdot f_t.
\end{eqnarray}
Substituting (3.8) and (3.9) into \eqref{37} derives
\begin{eqnarray}\label{310}
\mathscr{L}F&=&2\gamma\Big[\nabla^2 f|^{2}+Ric_V(\nabla f, \nabla f)+\left\langle \nabla f,-\nabla(|\nabla f|^2+G)\right\rangle\Big]-\gamma'|\nabla f|^2\nonumber\\
&&+\alpha' f_t+\alpha(2\left\langle \nabla f,\nabla f_t\right\rangle+\partial_tG+\partial_f G\cdot f_t)-\alpha'G+\alpha\mathscr{L}(G)+\varphi'.
\end{eqnarray}
Observe that
\begin{equation}
	\left\langle\nabla f,\nabla F\right\rangle=\gamma\left\langle\nabla f,\nabla |\nabla f|^2\right\rangle-\alpha \left\langle\nabla f,\nabla f_t\right\rangle+\alpha \left\langle\nabla f,\nabla G\right\rangle,\nonumber
\end{equation}
i.e.
\begin{equation}\label{311}
-2\gamma\left\langle\nabla f,\nabla |\nabla f|^2\right\rangle+2\alpha \left\langle\nabla f,\nabla f_t\right\rangle=-2\left\langle\nabla f,\nabla F\right\rangle+2\alpha \left\langle\nabla f,\nabla G\right\rangle.
\end{equation}
Plugging \eqref{311} into \eqref{310}, we obtain
\begin{eqnarray}\label{312}
\mathscr{L}F&=&2\gamma|\nabla^2 f|^{2}+2\gamma Ric_V(\nabla f, \nabla f)-2\left\langle \nabla f,\nabla F\right\rangle-\gamma'|\nabla f|^2\nonumber\\
&&+2(\alpha-\gamma)\left\langle \nabla f,\nabla G\right\rangle+\alpha' f_t+\alpha(\partial_tG+\partial_f G\cdot f_t)\nonumber\\
&&-\alpha'G+\alpha\mathscr{L}(G)+\varphi'.
\end{eqnarray}
By differential chain rule $\nabla G=\nabla_x G+\partial_f G\cdot \nabla f$ and \eqref{312}, we get desirable \eqref{36}.
\end{proof}

For our special case: $H(x,t,u)=H(u)$, denote $h(f):=G=H(e^f)\cdot e^{-f}$. 
Then
\begin{eqnarray}
	\nabla_x G&=&0.\\
	\partial_t G&=&0.\\
	\partial_f G&=&h'(f).\\
	\mathscr{L}G&=&\mathscr{L}h\nonumber\\
	&=&h'(f)\Delta_{V}f+h''(f)|\nabla f|^2-h'(f)f_t.
\end{eqnarray}
Substituting (3.13)-(3.16) into \eqref{36}, we have the following lemma.
\begin{lem}\label{l32}
	Let $u$ be a positive solution of equation \eqref{oeq} on $\Omega\times(0,\infty)$ and $\alpha$, $\varphi$, $\gamma$ be defined as Lemma \ref{l31}. We set $f:=lnu$,
	$h(f):=H(e^f)\cdot e^{-f}$ and $F:=\gamma |\nabla f|^{2}-\alpha f_t+\alpha h-\varphi$, then
	\begin{eqnarray}
	\mathscr{L}F&=&2\gamma|\nabla^2 f|^2+2\gamma Ric_V(\nabla f,\nabla f)-2\left\langle \nabla f,\nabla F\right\rangle+[2(\alpha-\gamma)h'(f)-\gamma']\cdot|\nabla f|^2\nonumber\\
	&&+\alpha' f_t+\alpha(h'\Delta_V f+h''|\nabla f|^2)-\alpha' h+\varphi'.
	\end{eqnarray}		
\end{lem}

 Combining the trick in Cao \cite{CH}\footnote{The same trick was also used in Cao-Hamilton \cite{CH} and Li-Xu \cite{LX} a little later. Concretely, that is a observation as follows: $|A_{ij}|^2=|A_{ij}+c g_{ij}|^2-nc^2-2c\cdot tr(A)$, where $A$ is a covariant $2$-tensor and $c$ is a real function on a manifold.} and X.D.Li \cite{XL}, we have the following technical lemma.
\begin{lem}\label{kl}
 Let $u$, $f$, $h$, $F$ be defined as in Lemma \ref{l32} and $\gamma$, $\alpha$, $\varphi$, $c\in C^1(0,\infty)$ are some undetermined functions on time $t$. Then we have
 \begin{eqnarray}\label{318}
 	\mathscr{L}F&\ge& \frac{2\gamma}{m}(\Delta_{V} f+c)^2-2\left\langle \nabla f,\nabla F\right\rangle\nonumber\\
 	&&+\Big(\frac{4\gamma}{m}c+(\alpha-2\gamma)h'+\alpha h''-2K\gamma-\gamma'\Big)|\nabla f|^2\nonumber\\
 	&&+\frac{1}{\alpha}\Big(\frac{4\gamma}{m}c-\alpha h'-\alpha'\Big)(-\alpha f_t+\alpha h-\varphi)\nonumber\\
 	&&+\varphi'-\frac{2\gamma}{m}c^2+\frac{\varphi}{\alpha}\Big(\frac{4\gamma}{m}c-\alpha h'-\alpha'\Big).
 \end{eqnarray}
\end{lem}
\begin{proof}[\textbf{Proof}]
	By Lemma \ref{l32} and basic inequality $|\nabla^2 f|^2\ge \frac{1}{n}(\Delta f)^2$, we have
	 \begin{eqnarray}
	\mathscr{L}F&\ge&\frac{2\gamma}{n}(\Delta f)^2+2\gamma Ric_V(\nabla f,\nabla f)-2\left\langle \nabla f,\nabla F\right\rangle+[2(\alpha-\gamma)h'(f)-\gamma']\cdot|\nabla f|^2\nonumber\\
	&&+\alpha' f_t+\alpha(h'\Delta_V f+h''|\nabla f|^2)-\alpha' h+\varphi'.\nonumber\\
	&=&\frac{2\gamma}{n}(\Delta_V f-\left\langle V,f\right\rangle)^2+2\gamma Ric_V(\nabla f,\nabla f)-2\left\langle \nabla f,\nabla F\right\rangle\nonumber\\
	&&+[2(\alpha-\gamma)h'(f)-\gamma']\cdot|\nabla f|^2+\alpha' f_t+\alpha(h'\Delta_V f+h''|\nabla f|^2)-\alpha' h+\varphi'\nonumber\\
	&\ge&\frac{2\gamma}{n}\left(\frac{(\Delta_{V} f)^2}{\frac{m}{n}}-\frac{\left\langle V,\nabla f\right\rangle^2}{\frac{m}{n}-1}\right)+2\gamma Ric_V(\nabla f,\nabla f)-2\left\langle \nabla f,\nabla F\right\rangle\nonumber\\
	&&+[2(\alpha-\gamma)h'(f)-\gamma']\cdot|\nabla f|^2+\alpha' f_t+\alpha(h'\Delta_V f+h''|\nabla f|^2)-\alpha' h+\varphi'\nonumber\\
		&=&\frac{2\gamma}{m}(\Delta_{V} f)^2+2\gamma Ric_V^m(\nabla f,\nabla f)-2\left\langle \nabla f,\nabla F\right\rangle+[2(\alpha-\gamma)h'(f)-\gamma']\cdot|\nabla f|^2\nonumber\\
	&&+\alpha' f_t+\alpha(h'\Delta_V f+h''|\nabla f|^2)-\alpha' h+\varphi'\nonumber\\
	&=&\frac{2\gamma}{m}(\Delta_{V} f+c)^2-\frac{2\gamma c^2}{m}-\frac{4\gamma}{m}c\cdot \Delta_{V} f+2\gamma Ric_V^m(\nabla f,\nabla f)-2\left\langle \nabla f,\nabla F\right\rangle\nonumber\\
	&&+[2(\alpha-\gamma)h'(f)-\gamma']|\nabla f|^2+\alpha' f_t+\alpha(h'\Delta_V f+h''|\nabla f|^2)-\alpha' h+\varphi'\nonumber\\
	&=&\frac{2\gamma}{m}(\Delta_{V} f+c)^2-\frac{2\gamma c^2}{m}-\Big(\frac{4\gamma}{m} c-\alpha h'\Big)\Delta_{V} f+2\gamma Ric_V^m(\nabla f,\nabla f)-2\left\langle \nabla f,\nabla F\right\rangle\nonumber\\
	&&+\left[2(\alpha-\gamma)h'(f)+\alpha h''-\gamma'\right]|\nabla f|^2+\alpha' f_t-\alpha' h+\varphi'\nonumber\\
	&=&\frac{2\gamma}{m}(\Delta_{V} f+c)^2-\Big(\frac{4\gamma}{m} c-\alpha h'\Big)(f_t-|\nabla f|^2-h)+2\gamma Ric_V^m(\nabla f,\nabla f)-2\left\langle \nabla f,\nabla F\right\rangle\nonumber\\
	&&+\left[2(\alpha-1)h'(f)+\alpha h''-\gamma'\right]|\nabla f|^2+\alpha' f_t-\alpha' h+\varphi'-\frac{2\gamma c^2}{m}\nonumber\\
	&\ge&\frac{2\gamma}{m}(\Delta_{V} f+c)^2-\Big(\frac{4\gamma}{m} c-\alpha h'\Big)(f_t-|\nabla f|^2-h)-2\left\langle \nabla f,\nabla F\right\rangle\nonumber\\
	&&+\left[2(\alpha-\gamma)h'(f)+\alpha h''-2\gamma K-\gamma'\right]|\nabla f|^2+\alpha' f_t-\alpha' h+\varphi'-\frac{2\gamma c^2}{m}\nonumber\\
	&=& \frac{2\gamma}{m}(\Delta_{V} f+c)^2-2\left\langle \nabla f,\nabla F\right\rangle+\left(\frac{4\gamma}{m}c+(\alpha-2\gamma)h'+\alpha h''-2\gamma K-\gamma'\right)|\nabla f|^2\nonumber\\
	&&+\left(\frac{4\gamma}{m}c-\alpha h'-\alpha'\right)(-f_t+h)+\varphi'-\frac{2\gamma}{m}c^2\nonumber\\
	&=&\frac{2\gamma}{m}(\Delta_{V} f+c)^2-2\left\langle \nabla f,\nabla F\right\rangle+\left(\frac{4\gamma}{m}c+(\alpha-2\gamma)h'+\alpha h''-2\gamma K-\gamma'\right)|\nabla f|^2\nonumber\\
	&&+\frac{1}{\alpha}\Big(\frac{4\gamma}{m}c-\alpha h'-\alpha'\Big)(-\alpha f_t+\alpha h-\varphi)\nonumber\\
	&&+\varphi'-\frac{2\gamma}{m}c^2+\frac{\varphi}{\alpha}\left(\frac{4\gamma}{m}c-\alpha h'-\alpha'\right). \nonumber
	\end{eqnarray}
	Here we use the following basic inequality for second inequality above.
	\begin{equation}
		(a-b)^2\ge\frac{a^2}{t}-\frac{b^2}{t-1}\qquad \text{for} \quad t>1.\nonumber
	\end{equation}
	Then we complete our proof.	
\end{proof}

By Lemma \ref{kl} and the definition of $A_1$-system, we get the following lemma immediately.
\begin{lem}\label{l34}
	All notations are as same as in Lemma \ref{kl} and $\gamma$, $\alpha$, $\varphi$, $c$, $h$, $f$ satisfy $A_1$-system
	 {\rm(when $\mathbf{M^{n}}$ is compact)}. 
	Then we have
	\begin{eqnarray}
	\mathscr{L}F&\ge& \frac{2\gamma}{m}(\Delta_{V} f+c)^2-2\left\langle \nabla f,\nabla F\right\rangle+\frac{F}{\alpha}\left(\frac{4\gamma}{m}c-\alpha h'-\alpha'\right).\nonumber
	\end{eqnarray}		
\end{lem}
\begin{proof}[\textbf{Proof}]
	By the definition of $A_1$-system, we have
	\begin{equation}\label{319}
		\frac{4\gamma}{m}c+(\alpha-2\gamma)h'+\alpha h''-2K\gamma-\gamma'\ge\frac{\gamma}{\alpha}\left(\frac{4\gamma}{m}c-\alpha h'-\alpha'\right).
	\end{equation}
	\begin{equation}\label{320}
		\varphi'-\frac{2\gamma}{m}c^2+\frac{\varphi}{\alpha}\left(\frac{4\gamma}{m}c-\alpha h'-\alpha'\right)\ge 0.
	\end{equation}
	Substituting \eqref{319} and \eqref{320} into \eqref{318}, we get
	 \begin{eqnarray}
	\mathscr{L}F&\ge& \frac{2\gamma}{m}(\Delta_{V} f+c)^2-2\left\langle \nabla f,\nabla F\right\rangle+\frac{\gamma}{\alpha}\left(\frac{4\gamma}{m}c-\alpha h'-\alpha'\right)|\nabla f|^2\nonumber\\
	&&+\frac{1}{\alpha}\left(\frac{4\gamma}{m}c-\alpha h'-\alpha'\right)(-\alpha f_t+\alpha h-\varphi)\nonumber\\
	&=&\frac{2\gamma}{m}(\Delta_{V} f+c)^2-2\left\langle \nabla f,\nabla F\right\rangle+\frac{F}{\alpha}\left(\frac{4\gamma}{m}c-\alpha h'-\alpha'\right).\nonumber
	\end{eqnarray}
\end{proof}

\noindent\textbf{Remark 3.1} It is obvious that Lemma \ref{l34} is also valid for $A_2$ and $A_3$-system because their definitions both contain \eqref{319} and \eqref{320}.

\subsection{\textbf{The Proof of Theorem \ref{CGD}}}
In this subsection, we will prove differential Harnack inequalities on closed manifolds or compact manifolds with convex boundaries.

\begin{proof}[\textbf{Proof of Theorem \ref{CGD}}]
	Case1: $\mathbf{M^{n}}$ is a closed manifold.\\
	We define $F_0=\gamma|\nabla f|^2-\alpha f_t+\alpha h$. For fixed $T>0$, we denote $M(T)=\max\limits_{\mathbf{M^n}\times \left[0,T\right]}|F_0|$. By the definition of $A_1$ system, we can choose $\epsilon_0\in(0,T)$ such that $-\varphi\le -M(T)$ in $(0,\epsilon_0]$. Hence $F\le 0$ on $\mathbf{M^n}\times (0,\epsilon_0]$. 
	
	On $\mathbf{M^n}\times [\epsilon_0, T]$, by Lemma \ref{l34} and the definition of $A_2$-system, we have
\begin{eqnarray}\label{321}
\mathscr{L}F+2\left\langle \nabla f,\nabla F\right\rangle> 0\qquad \text{on}\quad \{x\in\mathbf{M^{n}}:F>0\}.
\end{eqnarray}
Applying the weak maximal principle on closed manifold (see \cite{CLN}), we get $F\le 0$ on $\mathbf{M^n}\times [\epsilon_0, T]$. Hence $F\le0$ on $\mathbf{M^n}\times (0, T]$.

Case2: $\mathbf{M^{n}}$ is a compact manifold with convex boundry. As in Case1, if $F$ get its positive maximal value at ($x_0$,$t_0$), then $t_0>\epsilon_0$. Then $x_0\in\partial \mathbf{M}$ or \eqref{321} will be invalid at that point. Then we follow standard argument as follows: By strong maximal principle we have $\frac{\partial F}{\partial \textbf{n}}(x_0,t_0)>0$, here $\textbf{n}$ is the unit outer normal vector filed. If we choose an orthonormal basis $(e_i)_{1\le i\le n}$ for $TM$ at a neighborhood of $x_0$, where $\textbf{n}
=e_n$. Direct computation gives:
\begin{eqnarray}\label{322}
	F_n&=&2\gamma\sum_{i=1}^{n-1}f_jf_{jn}+2\gamma f_{n}f_{nn}+\alpha h' f_{n}-\alpha f_{tn}\nonumber\\
	&=&2\gamma\sum_{j=1}^{n-1}f_jf_{jn}\nonumber\\
	&=&2\gamma\sum_{i,j=1}^{n-1}-A_{ij}f_if_{j}\nonumber\\
	&\le&0.
\end{eqnarray}
Here $A_{ij}$ are components of the second fundamental form of $\partial\mathbf{M^{n}}$. The second equality is due to Neumann boundary condition, the third is due to Weingarten fomula and the last inequality is due to the convexity of boundry. Hence \eqref{322} yields a contradiction.

This ends the proof of Theorem \ref{CGD}.
\end{proof}

%%%%%%%%%%%%%%%%%
\section{\textbf{The proof of Theorem \ref{GD}}}
In this section, we will generalize the differential Harnack inequalities of equation \eqref{oeq} to complete Riemannian manifolds using the cut-off function technique from \cite{LY}. In order to obtain the cut-off function, we will utilize Bakry-Qian's generalized Laplacian comparison theorem\footnote{Bakry-Qian \cite{BQ2} and X.D.Li \cite{XL} independently proved this theorem by different approaches.} (see \cite{BQ2,XL}). In the proof of Theorem \ref{GD}, we will employ an argument based on \cite{LY} to solve the quadratic inequality associated with the auxiliary function.

First, by equation \eqref{oeq} and same notations are same as in Section 3, we have

\begin{equation}
	f_t=\Delta_{V} f+|\nabla f|^2+h(f),
\end{equation}
where $h(f)=H(e^f)\cdot e^{-f}$. Hence
\begin{eqnarray}
	\Delta_{V} f+c&=&f_t-|\nabla f|^2-h+c\nonumber\\
	&=&\frac{\gamma|\nabla f|^2-\varphi-F}{\alpha}-|\nabla f|^2+c\nonumber\\
	&=&\frac{\gamma-\alpha}{\alpha}|\nabla f|^2-\frac{\varphi}{\alpha}+c-\frac{F}{\alpha}.
\end{eqnarray}

As in Li-Yau \cite{LY}, we choose the cut-off function $\phi$ such that:\\
$(1)$ $supp\phi\subset B(x_0,2R)$ and $\phi=1$ on $B(x_0,R)$.\\
$(2)$ $\frac{|\nabla\phi|}{\phi}\le\frac{C}{R}$.\\
$(3)$ $|\Delta_{V}\phi|\le\frac{C}{R^2}\left(1+\sqrt{K}R\coth(\sqrt{KR})\right)$ on $\mathbf{M^{n}}\setminus Cut(x_0)$.\\
\textbf{Remark 4.1} The conclusion (3) depends on the condition $Ric_V^m\ge-Kg$ and Bakry-Qian's Laplacian comprison theorem (\cite[Theorem 1 and Corollary 2]{XL}). For the completeness, we provide the construction of $\phi$ as follows: first, we choose a function $m\in C^{\infty}([0,\infty))$ such that $m=1$ in $[0,1]$, $supp(m)\subset[0,2]$, $|m'|\le 2$ and $|m''|\le C$. Then the function 
\begin{equation}
	\phi(x)=m\Big(\frac{d(x,x_0)}{R}\Big)
\end{equation}
obviously satisfies (1) and (2). For (3), we have
\begin{eqnarray}
     |\Delta_{V}\phi(x)|&=&|m'\cdot\frac{1}{R}\Delta_{V}d+m''\cdot \frac{1}{R^2}|\nabla d|^2|\nonumber\\
     &\le&\frac{2}{R}|\Delta_{V}d|+\frac{C}{R^2}\nonumber\\
     &\le&\frac{C}{R}\sqrt{K}\coth(\sqrt{KR})+\frac{C}{R^2}
\end{eqnarray}
outside of cut locus of $x_0$, here the last inequality is due to Bakry-Qian's generalized Laplacian comprison theorem.

%\subsection{\textbf{The proof of main result}}

Now, we start to prove Theorem 2.2. 
\begin{proof}[\textbf{Proof of Theorem 2.2}]  We fix a positive time $T$. Let $G:=\beta\phi F$ be the auxiliary function. We assume $G$ achieves its maximum at $(x,s)$ on $\mathbf{M^{n}}\times[0,T]$. Without loss of generality, we may assmue that $G(x,s)>0$ (hence $s>0$ by $\beta(0)=0$) and $x\notin Cut(x_0)$ (or we can use Calabi's argument, see \cite{LY}). At the point $(x,s)$, it follows that
\begin{eqnarray}
	0\ge \mathscr{L}G&=&\mathscr{L}(\beta\phi)\cdot F+\mathscr{L}F\cdot \beta\phi+2\beta\left\langle\nabla \phi,\nabla F\right\rangle\nonumber\\
	&=&(\beta\Delta_V \phi-\beta'\phi)F+\mathscr{L}F\cdot \beta\phi+2\beta\left\langle\nabla \phi,\frac{\nabla(\phi F)-\nabla\phi\cdot F}{\phi}\right\rangle\nonumber\\
	&=&-\frac{\beta'}{\beta}G+\frac{\Delta_{V}\phi}{\phi}G-2\frac{|\nabla \phi|^2}{\phi^2}\cdot G+\mathscr{L}F\cdot \beta\phi.
\end{eqnarray}
By Remark 3.1, Lemma 3.4 and (4.2), we have
\begin{eqnarray}
\mathscr{L}F\ge\frac{2\gamma}{m\alpha^2}\Big[(\gamma-\alpha)|\nabla f|^2-\varphi+\alpha c-F\Big]^2
-2\left\langle \nabla f,\nabla F\right\rangle+\frac{F}{\alpha}\left(\frac{4\gamma}{m}c-\alpha h'-\alpha'\right).\nonumber\\
\end{eqnarray}
Plugging (4.6) into (4.5) yields
\begin{eqnarray}
0&\ge&-\frac{\beta'}{\beta}G+\frac{\Delta_{V}\phi}{\phi}G-2\frac{|\nabla \phi|^2}{\phi^2}\cdot G+\frac{2\gamma\beta\phi}{m\alpha^2}\Big[(\gamma-\alpha)|\nabla f|^2-\varphi+\alpha c-F\Big]^2\nonumber\\
&&-2\beta\phi\left\langle \nabla f,\nabla F\right\rangle
+\frac{G}{\alpha}\Big(\frac{4\gamma}{m}c-\alpha h'-\alpha'\Big)\nonumber\\
&=&-\frac{\beta'}{\beta}G+\frac{\Delta_{V}\phi}{\phi}G-2\frac{|\nabla \phi|^2}{\phi^2}\cdot G+\frac{2\gamma\beta\phi}{m\alpha^2}\Big[(\gamma-\alpha)|\nabla f|^2-\varphi+\alpha c-F\Big]^2\nonumber\\
&&+2G\left\langle \nabla f,\frac{\nabla \phi}{\phi}\right\rangle
+\frac{G}{\alpha}\Big(\frac{4\gamma}{m}c-\alpha h'-\alpha'\Big)\nonumber\\
&\ge&-\frac{\beta'}{\beta}G+\frac{\Delta_{V}\phi}{\phi}G-2\frac{|\nabla \phi|^2}{\phi^2}\cdot G+\frac{2\gamma\beta\phi F^2}{m\alpha^2}+\frac{4\gamma\beta\phi F}{m\alpha^2}\Big((\alpha-\gamma)|\nabla f|^2+\varphi-\alpha c\Big)\nonumber\\
&&+2G\left\langle \nabla f,\frac{\nabla \phi}{\phi}\right\rangle
+\frac{G}{\alpha}\Big(\frac{4\gamma}{m}c-\alpha h'-\alpha'\Big)\nonumber\\
&=&-\frac{\beta'}{\beta}G+\frac{\Delta_{V}\phi}{\phi}G-2\frac{|\nabla \phi|^2}{\phi^2}\cdot G+\frac{2\gamma G^2}{m\alpha^2\beta\phi}+\frac{4\gamma G}{m\alpha^2}\Big((\alpha-\gamma)|\nabla f|^2+\varphi-\alpha c\Big)\nonumber\\
&&+2G\left\langle \nabla f,\frac{\nabla \phi}{\phi}\right\rangle
+\frac{G}{\alpha}\Big(\frac{4\gamma}{m}c-\alpha h'-\alpha'\Big).\nonumber\\
\end{eqnarray}
The first equality above uses the following fact: 
\begin{equation}
-2\beta\phi\left\langle \nabla f,\nabla F\right\rangle=-2\left\langle \nabla f,\nabla  (\beta\phi F)-\beta F\nabla \phi \right\rangle=2G\left\langle \nabla f,\frac{\nabla \phi}{\phi}\right\rangle\nonumber
\end{equation}
at the maximum point.

Multiplying the $\frac{m\alpha^2\beta\phi}{2\gamma G}$ to both sides of (4.7), we obtain
\begin{align}
	G\le&\frac{m\alpha^2\beta\phi}{2\gamma}\Big[\frac{\beta'}{\beta}-\frac{\Delta_{V}\phi}{\phi}+2\frac{|\nabla \phi|^2}{\phi^2}-\frac{4\gamma }{m\alpha^2}\Big((\alpha-\gamma)|\nabla f|^2+\varphi-\alpha c\Big)\nonumber\\
	&-\frac{2}{\phi}\left\langle \nabla f,\nabla \phi\right\rangle-\frac{1}{\alpha}\Big(\frac{4\gamma}{m}c-\alpha h'-\alpha'\Big)\Big]\nonumber\\
	=&\frac{m\alpha^2\beta\phi}{2\gamma}\Big[\frac{\beta'}{\beta}+\frac{\alpha'}{\alpha}+h'-\frac{4\varphi\gamma}{m\alpha^2}-\frac{\Delta_{V}\phi}{\phi}+2\frac{|\nabla \phi|^2}{\phi^2}\nonumber\\
	&-\frac{4\gamma }{m\alpha^2}(\alpha-\gamma)|\nabla f|^2-\frac{2}{\phi}\left\langle \nabla f,\nabla \phi\right\rangle\Big]\nonumber\\
	\le&\frac{m\alpha^2\beta}{2\gamma}\Big[-\Delta_{V}\phi+2\frac{|\nabla \phi|^2}{\phi}+2|\nabla \phi||\nabla f|-\frac{4\gamma\phi }{m\alpha^2}(\alpha-\gamma)|\nabla f|^2\Big]\nonumber\\
	&+\frac{m\alpha^2\beta\phi}{2\gamma}\Bigg(h'+\frac{\alpha'}{\alpha}+\frac{\beta'}{\beta}-\frac{4\gamma\varphi}{m\alpha^2}\Bigg).
\end{align}
Due to $\alpha>\gamma$, we have
\begin{equation}
	2|\nabla \phi||\nabla f|\le \frac{4\gamma\phi }{m\alpha^2}(\alpha-\gamma)|\nabla f|^2+\frac{m\alpha^2 }{4\gamma(\alpha-\gamma)}\frac{|\nabla \phi|^2}{\phi}.
\end{equation}
Hence 
\begin{align}
	G\le&\frac{m\alpha^2\beta}{2\gamma}\Big[-\Delta_{V}\phi+2\frac{|\nabla \phi|^2}{\phi}+\frac{m\alpha^2}{4\gamma(\alpha-\gamma)} \frac{|\nabla \phi|^2}{\phi}\Big]\nonumber\\
	&+\frac{m\alpha^2\beta\phi}{2\gamma}\Bigg(h'+\frac{\alpha'}{\alpha}+\frac{\beta'}{\beta}-\frac{4\gamma\varphi}{m\alpha^2}\Bigg).
\end{align}
Combining (4.10) and the definition of $A_3$-system, we have
\begin{align}
	G(x,s)\le&\frac{m\alpha^2\beta}{2\gamma}\Big[|\Delta_{V}\phi|+2\frac{|\nabla \phi|^2}{\phi}+\frac{m\alpha^2 }{4\gamma(\alpha-\gamma)}\frac{|\nabla \phi|^2}{\phi}\Big]\nonumber\\
	\le&\frac{m\alpha^2\beta}{2\gamma}\Big[\frac{C}{R^2}(1+\sqrt{K}R\coth(\sqrt{K}R))+\frac{m\alpha^2}{4\gamma(\alpha-\gamma)}\cdot\frac{C}{R^2}\Big]\nonumber\\
	\le &
	\begin{cases}
	\frac{m\alpha^2\beta(T)}{2\epsilon}\Big[\frac{C}{R^2}(1+\sqrt{K}R\coth(\sqrt{K}R))+\frac{m\alpha^2(T)}{4\epsilon^2}\cdot\frac{C}{R^2}\Big]\quad\,\,\,\,	\text{if}\quad (I)\quad\text{holds}\vspace{2ex}\nonumber\\
	\frac{m\alpha^2\beta}{2\gamma}(T)\cdot\frac{C}{R^2}(1+\sqrt{K}R\coth(\sqrt{K}R))+m^2\alpha^4(T)\frac{C}{R^2}\quad\,\,	\text{if}\quad (II)\quad\text{holds}\vspace{2ex}\nonumber\\
	\frac{m\alpha^2\beta}{2\gamma}(T)\cdot\frac{C}{R^2}(1+\sqrt{K}R\coth(\sqrt{K}R))+\frac{m^2\alpha^4}{\gamma^2}(T)\cdot\frac{C}{R^2}\quad	\text{if}\quad (III)\quad\text{holds}.\nonumber
	\end{cases}
\end{align}
Hence for any $y\in B_{x_0}(R)$, we get
\begin{align}
	\beta(T)F(y,T)\le&\sup_{z\in B_{x_0}(R)}\beta(T)\phi(z)F(z,T)\nonumber\\
	\le&G(x,s)\nonumber\\
	\le &
	\begin{cases}
		\frac{m\alpha^2\beta(T)}{2\epsilon}\Big[\frac{C}{R^2}(1+\sqrt{K}R\coth(\sqrt{K}R))+\frac{m\alpha^2(T)}{4\epsilon^2}\cdot\frac{C}{R^2}\Big]\quad	\,\,\,\,\text{if}\quad (I)\quad\text{holds}\vspace{2ex}\nonumber\\
		\frac{m\alpha^2\beta}{2\gamma}(T)\cdot\frac{C}{R^2}(1+\sqrt{K}R\coth(\sqrt{K}R))+m^2\alpha^4(T)\frac{C}{R^2}\quad\,\,	\text{if}\quad (II)\quad\text{holds}\vspace{2ex}\nonumber\\
		\frac{m\alpha^2\beta}{2\gamma}(T)\cdot\frac{C}{R^2}(1+\sqrt{K}R\coth(\sqrt{K}R))+\frac{m^2\alpha^4}{\gamma^2}(T)\cdot\frac{C}{R^2}\quad	\text{if}\quad (III)\quad\text{holds},\nonumber
	\end{cases}
\end{align}
i.e.

\begin{align}
	F(y,T)\le
	\begin{cases}
		\frac{m\alpha^2(T)}{2\epsilon}\Big[\frac{C}{R^2}(1+\sqrt{K}R\coth(\sqrt{K}R))+\frac{m\alpha^2(T)}{4\epsilon^2}\cdot\frac{C}{R^2}\Big]\quad\,\,\,\,	\text{if}\quad (I)\quad\text{holds}\vspace{2ex}\nonumber\\
		\frac{m\alpha^2}{2\gamma}(T)\cdot\frac{C}{R^2}(1+\sqrt{K}R\coth(\sqrt{K}R))+m^2\frac{\alpha^4}{\beta}(T)\frac{C}{R^2}\quad\,\,	\text{if}\quad (II)\quad\text{holds}\vspace{2ex}\nonumber\\
		\frac{m\alpha^2}{2\gamma}(T)\cdot\frac{C}{R^2}(1+\sqrt{K}R\coth(\sqrt{K}R))+\frac{m^2\alpha^4}{\beta\gamma^2}(T)\cdot\frac{C}{R^2}\quad	\text{if}\quad (III)\quad\text{holds}.\nonumber
	\end{cases}
\end{align}
Because $T$ is arbitrary, we get the local gradient estimate. For global estimate, we let $R\to\infty$, then
\begin{align}
	F(y,T)\le 0
\end{align}
 holds for any $(y,T)\in \mathbf{M^{n}}\times(0,\infty)$. This ends the proof of Theorem \ref{GD}.
\end{proof}

\section{\textbf{Application to logarithmic type equation}}
In this section, we will consider the global differential Harnack inequality for logarithmic-type equations. In particular, we have proved sharp differential Harnack inequality, Harnack inequality, and Liouville-type theorem under the condition of non-negative Bakry-\'{E}mery curvature.

As before, for logarithmic type equation, we have
\begin{equation}
	h(f):=\frac{a\cdot e^f\cdot \ln(e^f)}{e^f}=a f.
\end{equation}
Hence solving $A_3$-system becomes
\begin{align}
	\begin{cases}
		\frac{4\gamma}{m}c+a(\alpha-\gamma)-2K\gamma-\gamma'\ge \frac{\gamma}{\alpha}\Big(\frac{4\gamma}{m}c-\alpha'\Big)\nonumber\\	
		\varphi'-\frac{2\gamma}{m}c^2+\frac{\varphi}{\alpha}\Big(\frac{4\gamma}{m}c-a\alpha-\alpha'\Big)\ge 0\nonumber\\
		a+\frac{\alpha'}{\alpha}+\frac{\beta'}{\beta}-\frac{4\gamma\varphi}{m\alpha^2}\le 0\nonumber\\
		\beta(0)=0\quad\text{and}\quad\beta >0\nonumber
	\end{cases}
\end{align}
and any one of (I)-(III) holds.

In this section, we assume $K=0$.

\begin{proof}[\textbf{Proof of Theorem 2.3}]
We divide our argument into compact and noncompact case.\\
Case1: $\mathbf{M^{n}}$ is a closed manifold or compact manifold with convex boundry at case $a\neq0$. In this case, we require to solve the easier $A_2$-system:
\begin{align}
	\begin{cases}
		\frac{4\gamma}{m}c+(\alpha-\gamma)a-\gamma'\ge \frac{\gamma}{\alpha}\Big(\frac{4\gamma}{m}c-\alpha'\Big)\nonumber\\
		\varphi'-\frac{2\gamma}{m}c^2+\frac{\varphi}{\alpha}\Big(\frac{4\gamma}{m}c-\alpha a-\alpha'\Big)\ge 0\nonumber\\
		\frac{4\gamma}{m}c-\alpha a-\alpha'>0\quad\text{and}\quad\gamma>0\nonumber
	\end{cases}
\end{align}
and
\begin{align}
	\begin{cases}
		\lim\limits_{t\to0^{+}}\alpha \qquad\text{exists}\\
		\lim\limits_{t\to0^{+}}\gamma \qquad\text{exists}\\
		\lim\limits_{t\to0^{+}}\varphi=\infty.\nonumber
	\end{cases}
\end{align}
It is easy to see 
\begin{eqnarray}
\gamma(t)=1, \qquad \alpha(t)=1,\nonumber\\
\varphi(t)=\frac{ ma}{2(1-e^{-at})}\quad \text{and}\quad c(t)=\frac{ma}{2(1-e^{-at})}
\end{eqnarray}
satisfy $A_2$-system. By Theorem 2.1, we have the global estimate:
\begin{equation}
\frac{|\nabla u|^2}{u^2}-\frac{u_t}{u}+a\ln u\leq\frac{ ma}{2(1-e^{-at})}.
\end{equation}
Case2: $\mathbf{M^{n}}$ is complete manifold with $a>0$.\\
We choose
\begin{eqnarray}
\gamma(t)=1, \quad \alpha(t)=\alpha>1,\quad \beta(t)=1-e^{-at},\nonumber\\
\varphi(t)=\frac{\alpha^2\cdot ma}{2(1-e^{-at})},\qquad c(t)=\frac{\alpha\cdot ma}{2(1-e^{-at})}.
\end{eqnarray}
Then (5.4) solves $A_3$-system.
By Theorem 2.2 (satisfies (I)), we have the global estimate:
\begin{equation}
\frac{|\nabla u|^2}{u^2}-\alpha\frac{u_t}{u}+\alpha a\ln u\leq\frac{\alpha^2\cdot ma}{2(1-e^{-at})}.
\end{equation}
Letting $\alpha\to1^+$, we have
\begin{equation}
\frac{|\nabla u|^2}{u^2}-\frac{u_t}{u}+a\ln u\leq\frac{ ma}{2(1-e^{-at})}.
\end{equation}
Case3: $\mathbf{M^{n}}$ is complete manifold with $a<0$. 

In this case, if we are unable to find a solution on the interval $(0,\infty)$ for the $A_3$-system, we can instead focus on studying the asymptotic behavior of a family of solutions on a finite time interval. As mentioned in Remark 2.2, even if we only have the solution of the $A_3$-system on a finite time interval, we can still derive the corresponding differential Harnack inequalities on that interval. Using this simple but important observation, we can proceed with the following analysis argument.

First, we notice that for some undetermined positive function $\gamma$, the following solution 
\begin{eqnarray}
0<\gamma(t)<1, \quad \gamma\quad\text{is non-increasing,}\quad \alpha(t)=1,\quad \beta(t)=e^{-at}-1,\nonumber\\
\varphi(t)=\frac{1}{\gamma}\cdot\frac{ma}{2(1-e^{-at})}\quad\text{and}\quad c(t)=\frac{1}{\gamma}\cdot\frac{ma}{2(1-e^{-at})}
\end{eqnarray}
satisfies all inequalities in the $A_3$-system except the first inequality.

This fact could be obtained by direct computation as follows. 
Define $\varphi_0=\frac{ma}{2(1-e^{-at})}$, hence $\varphi(t)=c(t)=\frac{1}{\gamma}\cdot\varphi_0$. Then we have
\begin{eqnarray}
&&\varphi'-\frac{2\gamma}{m}c^2+\frac{\varphi}{\alpha}\Big(\frac{4\gamma}{m}c-a\alpha-\alpha'\Big)\nonumber\\
&=&(\frac{1}{\gamma})'\varphi_0+\frac{1}{\gamma}\varphi_0'-\frac{2}{m}\frac{\varphi_0^2}{\gamma}+\frac{4}{m}\frac{\varphi_0^2}{\gamma}-a\frac{\varphi_0}{\gamma}\nonumber\\
&=&(\frac{1}{\gamma})'\varphi_0\nonumber\\
&\ge& 0
\end{eqnarray}
and
\begin{eqnarray}
a+\frac{\alpha'}{\alpha}+\frac{\beta'}{\beta}-\frac{4\gamma\varphi}{m\alpha^2}=\frac{a}{e^{-at}-1}\le 0.
\end{eqnarray}
Then the first inequality of $A_3$-system becomes:
\begin{equation}
\frac{4\varphi_0}{m}(1-\gamma)+a(1-\gamma)-\gamma'\ge 0,
\end{equation}
i.e.
\begin{equation}
\frac{4\varphi_0}{m}+a+(\frac{1}{\gamma})'\frac{\gamma^2}{1-\gamma}\ge 0.
\end{equation}
Define $k=\frac{1}{\gamma}$, $l=k-1>0$. Then
\begin{eqnarray}
\text{(5.11)}&&\Longleftrightarrow\frac{4\varphi_0}{m}+a+\frac{k'}{k(k-1)}\ge 0.\nonumber\\
&&\Longleftrightarrow\frac{4\varphi_0}{m}+a+\frac{l'}{l(l+1)}\ge 0.\nonumber\\
&&\Longleftrightarrow\Big(\frac{4\varphi_0}{m}+a\Big)l(l+1)+l'\ge 0.
\end{eqnarray}
Therefore, we can define $l$ as follows:
\begin{eqnarray}
l(t)=\nonumber
\begin{cases}
\epsilon \qquad\qquad\qquad\qquad\,\text{if} \quad 0<t\le\frac{\ln 3}{-a}\nonumber\\
m(t-\frac{\ln 3}{-a})\qquad\qquad\text{if}  \quad t\ge\frac{\ln 3}{-a}.
\end{cases}
\end{eqnarray}
Here $\epsilon>0$ and $m$ is the unique solution of the following Cauchy problem:
\begin{eqnarray}
\begin{cases}
m'(t)=m(m+1)(-a)\cdot\frac{3(e^{-at}-1)}{3e^{-at}-1}\vspace{2ex}\\
m(0)=\epsilon.
\end{cases}
\end{eqnarray}
We denote the solution of equation (5.13) as $m_\epsilon(t)$. Assuming that $(B_\epsilon,A_\epsilon)$ is the maximal existence interval of (5.13), it is easy to see that $B_\epsilon<0$.

\textbf{Claim 1}: $l$ is $C^1$ on $(0,A_\epsilon+\frac{\ln 3}{-a})$.

\textbf{Proof of Claim 1} By the ODE fundamental theorem, we can conclude that $m$ is smooth on $(B_\epsilon,A_\epsilon)$. Therefore, $l$ is also smooth on $(0,\frac{\ln 3}{-a}+A_\epsilon)$. Additionally, it is easy to see that $l$ is continuous, differentiable, and that $l'(t)$ is continuous as well. We can check the continuity of $l'$ at time $\frac{\ln3}{-a}$, which is zero, to verify this.

\textbf{Claim 2}: $\lim\limits_{\epsilon\to0^+}A_\epsilon=\infty$.

\textbf{Proof of Claim 2}   We devide the argument into three steps.

Step 1: $\delta_0:=\frac{1}{48(-a)}$, then (5.13) has a unique solution in $[-\delta_0,\delta_0]$ for any $\epsilon\in[-\frac{1}{4},\frac{1}{4}]$. 

We rewrite (5.13) as
\begin{eqnarray}
\begin{cases}
m'(t)=(m+\epsilon)(m+1+\epsilon)(-a)\cdot\frac{3(e^{-at}-1)}{3e^{-at}-1}\vspace{2ex}\\
m(0)=0.
\end{cases}
\end{eqnarray}
Then $m$ solves (5.14) if and only if $m+\epsilon$ solves (5.13).
We define 
\begin{equation}
g(t,m,\epsilon)=(m+\epsilon)(m+1+\epsilon)(-a)\cdot\frac{3(e^{-at}-1)}{3e^{-at}-1}\nonumber\\
\end{equation}
and region $\mathbf{R}=[\frac{\ln 2}{a},\frac{\ln2}{-a}]\times[-\frac{1}{4},\frac{1}{4}]\times[-\frac{1}{4},\frac{1}{4}]$.
Then we have\\
(i) $g$ is uniformly Lipschitz in $\mathbf{R}$ on variable $m$. In fact, 
\begin{eqnarray}
&&|g(t,m_1,\epsilon)-g(t,m_2,\epsilon)|\nonumber\\
&=&|(m_1+\epsilon)(m_1+1+\epsilon)-(m_1+\epsilon)(m_1+1+\epsilon)|(-a)\frac{3(e^{-at}-1)}{3e^{-at}-1}\nonumber\\
&=&|2\epsilon+1+2\xi|\cdot|m_1-m_2|(-a)\frac{3(e^{-at}-1)}{3e^{-at}-1}\nonumber\\
&\le&2(-a)\cdot6\cdot|m_1-m_2|,
\end{eqnarray}
where the second equality is due to the Lagrange mean value theorem.\\
(ii) Then $M\le12(-a)$, where $M:=\max_{(t,m,\epsilon)\in\mathbf{R}}|g(t,m,\epsilon)|$. ODE fundemental theorem gives that equation (5.14) has unique solution in $[-\eta_0,\eta_0]$, where\vspace{3mm}
\begin{equation}
\eta_0=\min\Big(\frac{\ln 2}{-a},\frac{\frac{1}{4}}{M}\Big)\ge\frac{1}{48(-a)}.\nonumber
\end{equation}
Hence Step 1 finishes.

Step 2:  $\delta_0$ is as in step1, $i\in\mathbb{N}$. Then the following Cauchy problem
\begin{eqnarray}
\begin{cases}
m'(t)=m(m+1)(-a)\cdot\frac{3(e^{-at}-1)}{3e^{-at}-1}\vspace{2ex}\\
m(i\delta_0)=m_i
\end{cases}
\end{eqnarray}
 has a unique solution in $[(i-1)\delta_0,(i+1)\delta_0]$ for any $m_i\in[-\frac{1}{4},\frac{1}{4}]$.
 
As same as in Step 1, (5.16) is equivalent to
\begin{eqnarray}
\begin{cases}
m'(t)=(m+m_i)(m+1+m_i)(-a)\cdot\frac{3(e^{-at}-1)}{3e^{-at}-1}\vspace{2ex}\\
m(i\delta_0)=0.
\end{cases}
\end{eqnarray}
Define 
\begin{equation}
g_i(t,m,m_i)=(m+m_i)(m+1+m_i)(-a)\cdot\frac{3(e^{-at}-1)}{3e^{-at}-1}\nonumber\\
\end{equation}
and region $\mathbf{R_i}=[(i-1)\delta_0,(i+1)\delta_0]\times[-\frac{1}{4},\frac{1}{4}]\times[-\frac{1}{4},\frac{1}{4}]$.
Same process as in Step 1 yields that (5.17) (hence (5.16)) has a unique solution in $[(i-1)\delta_0,(i+1)\delta_0]$ for any $m_i\in[-\frac{1}{4},\frac{1}{4}]$.

Step 3: For $T\in(0,\infty)$, $N:=\lceil\frac{T}{\delta_0}\rceil$ (minimal integer larger than $\frac{T}{\delta_0}$).

By Step 1, (5.13) has a unique solution in $[-\delta_0,\delta_0]$ for any $\epsilon\in[-\frac{1}{4},\frac{1}{4}]$. We denote the unique solution of (5.13) as $m(t,\epsilon)$. By ODE fundemental theorem again:
\begin{equation}
\lim\limits_{\epsilon\to0}\parallel m(\cdot,\epsilon)-m(\cdot,0)\parallel_{L^{\infty}[0,\delta_0]}=0.
\end{equation}
It is easy to see $m(t,0)\equiv0$. Then there exists $\epsilon_1>0$ such that if $\epsilon\in(-\epsilon_1,\epsilon_1)$, we have
\begin{equation}
\parallel m(\cdot,\epsilon)\parallel_{L^{\infty}[0,\delta_0]}\le\frac{1}{4}.
\end{equation}
Hence $m(\delta_0,\epsilon)\in[-\frac{1}{4},\frac{1}{4}]$. By Step 2, (5.16) has a unique solution in $[0,2\delta_0]$ for any $m_1\in[-\frac{1}{4},\frac{1}{4}]$. If we set $m_1=m(\delta_0,\epsilon)$, then the corresponding solution of (5.16) also solves (5.13) in $[0,\delta_0]$ by uniqueness. Therefore, (5.13) has a unique solution in $[0,2\delta_0]$ for any $\epsilon\in(-\epsilon_1,\epsilon_1)$. By ODE fundemental theorem again:
\begin{equation}
\lim\limits_{\epsilon\to0}\parallel m(\cdot,\epsilon)\parallel_{L^{\infty}[0,2\delta_0]}=0.
\end{equation}
Then there exists $\epsilon_2>0$ such that if $\epsilon\in(-\epsilon_2,\epsilon_2)$, we have
\begin{equation}
\parallel m(\cdot,\epsilon)\parallel_{L^{\infty}[0,2\delta_0]}\le\frac{1}{4}.
\end{equation}
Hence $m(2\delta_0,\epsilon)\in[-\frac{1}{4},\frac{1}{4}]$. By Step 2, (5.16) has a unique solution in $[\delta_0,3\delta_0]$ for any $m_2\in[-\frac{1}{4},\frac{1}{4}]$. If we set $m_2=m(2\delta_0,\epsilon)$, then the corresponding solution of (5.16) also solves (5.13) in $[\delta_0,2\delta_0]$ by uniqueness. Therefore, (5.13) has a  unique solution in $[0,3\delta_0]$ for any $\epsilon\in(-\epsilon_2,\epsilon_2)$. The same process can be repeated iteratively, and we can continue this process until we obtain a sequence of intervals $(-\epsilon_{N-1},\epsilon_{N-1})$, $(-\epsilon_{N-2},\epsilon_{N-2})$, ..., $(-\epsilon_1,\epsilon_1)$ such that for each $n=1,2,...,N-1$, equation (5.13) has a unique solution on the interval $[0,n\delta_0]$ for $\epsilon\in(-\epsilon_n,\epsilon_n)$. 
This implies that for $\epsilon\in(0,\epsilon_{N-1})$, the solution of equation (5.13) exists for all $t\in[0,N\delta_0]$. Therefore, we have $A_\epsilon \geq N\delta_0>T$, which completes the proof of Claim 2.

We begin by choosing $l$ as before, which guarantees that $\gamma$ is non-increasing and $\alpha-\gamma \geq \frac{\epsilon}{1+\epsilon}$ (satisfying condition (I)). As a result, we obtain the following differential Harnack inequality in the interval $(0,A_\epsilon+\frac{\ln 3}{-a})$:
\begin{equation}
	\frac{1}{l+1}\frac{|\nabla u|^2}{u^2}-\frac{u_t}{u}+a\ln u\leq(l+1)\cdot\frac{ma}{2(1-e^{-at})}.
\end{equation}
For any fixed $T>0$, there exists $\epsilon_0>0$ such that if $\epsilon\in(0,\epsilon_0)$, we have $A_\epsilon > T$. Hence, equation (5.22) holds on the interval $(0,T]$ for $\epsilon\in(0,\epsilon_0)$. By the ODE fundamental theorem, we can deduce that
\begin{equation}
	\lim_{\epsilon\to0^+}\| m(\cdot,\epsilon)\|_{L^{\infty}[0,T]}=0.
\end{equation}
In the interval $(0,\frac{\ln 3}{-a})$, equation (5.22) becomes
\begin{equation}
	\frac{1}{\epsilon+1}\frac{|\nabla u|^2}{u^2}-\frac{u_t}{u}+a\ln u\leq(\epsilon+1)\cdot\frac{ma}{2(1-e^{-at})}.\nonumber
\end{equation}
Taking the limit as $\epsilon \to 0^+$, we obtain
\begin{equation}
	\frac{|\nabla u|^2}{u^2}-\frac{u_t}{u}+a\ln u\le\frac{ma}{2(1-e^{-at})}.
\end{equation}
In the interval $[\frac{\ln 3}{-a},\frac{\ln 3}{-a}+T]$, equation (5.22) becomes
\begin{equation}
	\frac{1}{m(t+\frac{\ln 3}{a},\epsilon)+1}\frac{|\nabla u|^2}{u^2}-\frac{u_t}{u}+a\ln u\leq(m(t+\frac{\ln 3}{a},\epsilon)+1)\cdot\frac{ma}{2(1-e^{-at})}.\nonumber
\end{equation}
Taking the limit as $\epsilon \to 0^+$, and utilizing (5.23), we have
\begin{equation}
	\frac{|\nabla u|^2}{u^2}-\frac{u_t}{u}+a\ln u\le\frac{ma}{2(1-e^{-at})}.
\end{equation}
Combining equations (5.24) and (5.25), we see that the inequality holds in $(0,\frac{\ln 3}{-a}+T]$. Since $T$ is arbitrary, we conclude that the differential Harnack inequality (2.6) is obtained.

By equation (2.5), we can derive the following equivalent form of equation (5.25):
\begin{equation}
	\Delta_V f+\frac{ma}{2(1-e^{-at})}\ge 0,
\end{equation}
where $f=\ln u$.

Sharpness of (5.26) for any $a\ne 0$:
There exists a family of particular solutions of (2.5) in $\mathbb{R}^n\times\mathbb{R}\setminus\{0\}$, taking the form of 
\begin{equation}
u(x,t)=\exp\Big[-\frac{a\parallel x-x_0\parallel^2}{4(1-e^{-at})}-\frac{n}{2}e^{at}\ln|1-e^{-at}|+Ce^{at}\Big],
\end{equation}
where $x_0\in\mathbb{R}^n$, $C\in\mathbb{R}$ is an arbitrary constant.

Direct computation gives
\begin{equation}
\frac{|\nabla u|^2}{u^2}-\frac{u_t}{u}+a\ln u=-\Delta (\ln u)=\frac{na}{2(1-e^{-at})}.
\end{equation}
Here the first equality is due to equation (2.5). It is easy to see that if $V=0$, we can substitute $m$ into $n$ in (5.26). Hence our differetial Harnack is sharp.

\end{proof}
\noindent\textbf{Remark 5.1}\\
It is unknown the sharp differential Harnark inequality of equation (2.5) under general condition $Ric_V^m\ge-Kg$ such as the sharp form in hyperbolic space.

\begin{proof}[\textbf{Proof of Theorem 2.4}]
	Define the function $s:[t_1,t_2]\rightarrow\mathbb{R}$ as
	\begin{equation}
	s(t)=f(l(t),t),
	\end{equation}
	where $l\in \mathscr{P}$ and $\mathscr{P}$ denotes all smooth paths connecting $x_1$ and $x_2$. By the chain rule and global differential Harnack inequality (2.3), we get
	\begin{equation}
	s'(t)=f_t+\nabla f\cdot\frac{dl}{dt}
	\end{equation}
	and 
	\begin{eqnarray}
	\frac{d}{dt}(e^{-at}s(t))&=&e^{-at}\left(-af+f_t+\nabla f\cdot\frac{dl}{dt}\right)\nonumber\\
	&\ge&e^{-at}\left(\frac{\gamma |\nabla f|^2-\varphi}{\alpha}+\nabla f\cdot\frac{dl}{dt}\right)\nonumber\\
	&\ge&e^{-at}\left(\frac{-\varphi}{\alpha}-\frac{\alpha}{4\gamma}\left|\frac{dl}{dt}\right|^2\right).
	\end{eqnarray}
	Hence (2.7) holds by integrating both sides of (5.31) from $t_1$ to $t_2$.
	
	For (2.8), we need to prove: for $x_1$, $x_2\in \mathbf{M^{n}}$, $0<t_1<t_2$ and $a\ne 0$, 
	\begin{equation}
	\min\{\int_{t_1}^{t_2}e^{-at}\left|\frac{dl}{dt}\right|^2dt: l\in\mathscr{P}\}=\frac{a\cdot d(x_1,x_2)^2}{e^{at_2}-e^{at_1}}.
	\end{equation}
	If we define $\mathscr{P}_L=\{l\in\mathscr{P}:\text{length of }l=L\}$, we can derive 	
	\begin{equation}
	\min\{\int_{t_1}^{t_2}e^{-at}\left|\frac{dl}{dt}\right|^2dt: l\in\mathscr{P}_L\}=\frac{a\cdot L^2}{e^{at_2}-e^{at_1}}.
	\end{equation}	
	First, it is easy to see: if $l\in\mathscr{P}_L$,
	\begin{eqnarray}
	L^2&=&\Big(\int_{t_1}^{t_2}\left|\frac{dl}{dt}\right|dt\Big)^2\nonumber\\
	&\le&\int_{t_1}^{t_2}e^{-at}\left|\frac{dl}{dt}\right|^2dt\cdot\int_{t_1}^{t_2}e^{at}dt\nonumber\\
	&=&\int_{t_1}^{t_2}e^{-at}\left|\frac{dl}{dt}\right|^2dt\cdot\frac{e^{at_2}-e^{at_1}}{a},
	\end{eqnarray}	
	which yields 	
	\begin{equation}
	\min\{\int_{t_1}^{t_2}e^{-at}\left|\frac{dl}{dt}\right|^2dt: l\in\mathscr{P}_L.\}\ge\frac{a\cdot L^2}{e^{at_2}-e^{at_1}}.
	\end{equation}		
	Notice that the second inequality is due to Cauchy inequality, which is equality if and only if $\left|\frac{dl}{dt}\right|=ce^{at}$, where $c=\frac{a\cdot L}{e^{at_2}-e^{at_1}}$ by the definition of $\mathscr{P}_L$. Hence (5.35) is also an equality. By (5.33), one can see that (5.32) holds by taking $L=d(x_1,x_2)$ (due to $\mathbf{M^{n}}$ is complete). In this equality case of (5.32), $l$ is a minimized geodesic from $x_1$ to $x_2$ with a speed given by
	\begin{equation}
	\left|\frac{dl}{dt}\right|=\frac{a\cdot L}{e^{at_2}-e^{at_1}}\cdot e^{at}.
	\end{equation}
		For $a\ne 0$, if we choose the sharp differential Harnack inequality
	\begin{equation}
	\frac{|\nabla u|^2}{u^2}-\frac{u_t}{u}+a\ln u\le\frac{na}{2(1-e^{-at})},
	\end{equation}
 then we can derive the corresponding Harnack inequality from (2.8):
		\begin{eqnarray}
	e^{-at_2}f(x_2,t_2)-e^{-at_1}f(x_1,t_1)&\ge&\int_{t_1}^{t_2}-e^{-at}\cdot \varphi dt-\frac{1}{4}\cdot\frac{a\cdot d(x_1,x_2)^2}{e^{at_2}-e^{at_1}}\nonumber\\
	&=&-\frac{n}{2}\ln\Big(\frac{1-e^{-at_2}}{1-e^{-at_1}}\Big)-\frac{1}{4}\cdot\frac{a\cdot d(x_1,x_2)^2}{e^{at_2}-e^{at_1}}.
	\end{eqnarray}
	If we take the solution $u$ as (5.27) ($\mathbf{M^{n}}$=$\mathbb{R}^n$), for $x_1$, $x_2\in \mathbf{R^{n}}$, $0<t_1<t_2$ and $a\ne 0$, we have
	\begin{eqnarray}
	e^{-at_2}f(x_2,t_2)-e^{-at_1}f(x_1,t_1)
	&=&-\frac{a \parallel x_2-x_0\parallel^2}{4(e^{at_2}-1)}+\frac{a \parallel x_1-x_0\parallel^2}{4(e^{at_2}-1)}-\frac{n}{2}\ln\Big(\frac{1-e^{-at_2}}{1-e^{-at_1}}\Big)\nonumber\\
	&=&-\frac{n}{2}\ln\Big(\frac{1-e^{-at_2}}{1-e^{-at_1}}\Big)-\frac{a\parallel x_1-x_2\parallel^2}{4(e^{at_2}-e^{at_1})},
	\end{eqnarray}
	where $f=\ln u$. The second equality holds by choosing
	\begin{equation}
	x_0=\frac{(e^{at_2}-1)x_1-(e^{at_1 }-1)x_2}{e^{at_2}-e^{at_1}}.
	\end{equation}
	Hence, the Harnack inequality associated with \eqref{sldhk} is also sharp under condition $Ric\ge 0$.
	Then we end the proof of Theorem 2.4.	
\end{proof}

Next, we derive some Liouville type theorems by differential Harnark inequality \eqref{sldhk}.
\begin{proof}[\textbf{Proof of Theorem 2.5}]
	(1) First, for fixed $t_1<0$, we choose an undetermined $T_0>-t_1$.\\
	We define $\tilde{u}(x,t)=u(x,t-T_0)$ on $\mathbf{M^{n}}\times(-\infty,T_0)$.
	 Then $\tilde{u}$ solves the equation (2.5) on $\mathbf{M^{n}}\times(-\infty,T_0)$. By differential Harnack inequality \eqref{sldhk}, we have
	\begin{eqnarray}\label{541}
	\Delta_{V}(\ln u)(x,t_1)&=&\Delta_{V}(\ln \tilde{u})(x,t)\nonumber\\
	&\ge&-\frac{ma}{2(1-e^{-at})}\nonumber\\
	&=&-\frac{ma}{2(1-e^{-at_1-aT_0})},
	\end{eqnarray}	
	where $t:=t_1+T_0$.
	Letting $T_0\to\infty$ yields
	\begin{equation}
	\Delta_{V}(\ln u)(x,t_1)\ge 0.
	\end{equation}
	By weak maximal principle in closed manifold, we see that $\ln u(\cdot,t_1)$ is constant on $\mathbf{M^{n}}$. Then (2.5) becomes an ODE, directly solving it finishes the proof of (1).	\\
	(2) By (1), we immediately get  
	\begin{equation}
	\Delta_{V}u=u\Delta_{V}(\ln u)+\frac{|\nabla u|^2}{u}\ge 0,
	\end{equation}
	and so
	\begin{equation}
	\partial_tf\ge af,
	\end{equation}
	where $f=\ln u$. For $a<0$ and any fixed $x\in\mathbf{M^{n}}$, if $x\in S_3$, there exists $t_0\in(-\infty,0)$ such that $f(x,t_0)<0$. Then we consider the following Cauchy problem:
	\begin{eqnarray}
	\begin{cases}
	\dot{v}=av\\
	v(t_0)=f(x,t_0).
	\end{cases}
	\end{eqnarray}
	We immediately get a unique solution of (5.46) in $(-\infty,0)$ as follows:
	\begin{eqnarray}
	v(t)=f(x,t_0)e^{a(t-t_0)}.
	\end{eqnarray}
	By ODE comparison theorem, we have $f(x,t)\le f(x,t_0)e^{a(t-t_0)}$ in $(-\infty,t_0)$ which yields (2b). If $x\in S_1$, the same argument as above gives (2a). If $u$ has lower bound, then $S_3=\emptyset$, because any point in $S_3$ will decay to 0 as $t\to-\infty$ by (2b). Hence (2c) holds. (2d) is also trivail by (2b) and the same argument as before. If we choose solutions in (5.27), when $t\to-\infty$, we see
	\begin{equation}
	\begin{cases}
	u(x,\cdot)=e^{O(e^{at})}\quad\,\,\,\text{if}\quad C>0\nonumber\\
	u(x,\cdot)=e^{\frac{n}{2}}\quad\qquad\,\text{if}\quad C=0\nonumber\\
	u(x,\cdot)=e^{O(-e^{at})}\quad\text{if}\quad C<0.\\
	\end{cases}
	\end{equation}
	Moreover, at case $C\ge 0$, $u>1$ (i.e. $S_1=\mathbf{M^{n}}$); at case $C<0$, $S_3=\mathbf{M^{n}}$. Hence (2a)-(2d) are sharp on time growth control. \\
	(3) It is easy to see that (2.10) implies (3a) and the bound of superior limit. Therefore we only need to prove estimate (2.10). As the proof of (1), the inequality \eqref{541} holds. By letting $T_0 \to \infty$, we can conclude that $\Delta_V(\ln u) \geq -\frac{ma}{2}$. Therefore, we have $\partial_t f \geq -a\left(\frac{m}{2}+f\right)$, where $f = \ln u$. The remaining process is routine and can be carried out similarly to what is done in (2).\\
	(4) In this case, we directly get the following inequality by sharp differential Harnack for $a>0$ case and equation (2.5):
	\begin{equation}
	\partial_t f\ge\frac{ma}{2(e^{-at}-1)}+af,
	\end{equation}
	where $f=\ln u$. For $a>0$ and any fixed $x\in\mathbf{M^{n}}$, if $x\in Z_1$, there exists $t_n\in(0,\infty)$ and $t_n\to\infty$, $f(x,t_n)\ge\sigma>\frac{m}{2}$. Then we consider a sequence Cauchy problems:
	\begin{eqnarray}\label{548}
	\begin{cases}
	\dot{v}=\frac{ma}{2(e^{-at}-1)}+av\\
	v(t_n)=f(x,t_n).
	\end{cases}
	\end{eqnarray}
	We immediately get the unique solution of \eqref{548} in $(0,\infty)$ as follows.
	\begin{eqnarray}
	v(t)&=&e^{a(t-t_n)}f(x,t_n)+\frac{m}{2}e^{at}\ln\left(\frac{e^{-at_n}-1}{e^{-at}-1}\right)\nonumber\\
	&=&-\frac{m}{2}e^{at}\ln(1-e^{-at})+c(x,t_n)e^{at},
	\end{eqnarray}
	where
	\begin{equation}
	c(x,t_n)=\frac{m}{2}\ln(1-e^{-at_n})+e^{-at_n}f(x,t_n).
	\end{equation}
	If $c(x,t_n)\le 0$ for all $n\in\mathbb{N}$, by ODE comparison theorem, we have 
	\begin{equation}\label{551}
	f(x,t)\le-\frac{m}{2}e^{at}\ln(1-e^{-at})\quad\text{in}\quad(0,t_n).
	\end{equation}
	Letting $t_n\to \infty$, then \eqref{551} holds on $(0,\infty)$ which yields $\limsup\limits_{t\to\infty}f(x,t)\le\frac{m}{2}$. This controdicts with $x\in Z_1$. Hence there exists some $n$ such that $c(x,t_n)>0$. Then by choosing $t_0=t_n$ and $c=c(x,t_n)$, ODE comparison theorem again derives our desired.\\
	If $f\ge \delta>\frac{m}{2}$, as before, for undetermined $t_0\in (0,\infty)$, we have
	\begin{eqnarray}
	f(x,t)&\ge&-\frac{m}{2}e^{at}\ln(1-e^{-at})+c(x,t_0)e^{at}\nonumber\\
	&\ge&-\frac{m}{2}e^{at}\ln(1-e^{-at})+\left(\delta e^{-at_0}+\frac{m}{2}\ln (1-e^{-at_0})\right)\cdot e^{at}\nonumber\\
	&=&-\frac{m}{2}e^{at}\ln(1-e^{-at})+\left(\frac{\delta}{2}-\frac{m}{4}\right) e^{-at_0}e^{at}\nonumber\\
	&&+\Big[\left(\frac{\delta}{2}+\frac{m}{4}\right)e^{-at_0}+\frac{m}{2}\ln (1-e^{-at_0})\Big]\cdot e^{at}\nonumber
	\end{eqnarray}
	in $[t_0,\infty)$.
	Then we choose
	\begin{equation}
	t_0=t(\delta)= \frac{1}{a}\ln\left(\frac{\delta+\frac{m}{2}}{\delta-\frac{m}{2}}\right),
	\end{equation}
	and so
	\begin{equation}
	\left(\frac{\delta}{2}+\frac{m}{4}\right)e^{-at_0}+\frac{m}{2}\ln (1-e^{-at_0})> 0.
	\end{equation}
	Thus we complete the proof of (4a) by setting $c=c(\delta)=(\frac{\delta}{2}-\frac{m}{4}) e^{-at_0}$.
	
	For (4b), we first prove a little stronger claim:
	
	\textbf{Claim:} If $\liminf\limits_{t\to\infty}\ln u(x,t)\le \frac{m}{2}$,  we have $u(x,t)\le e^{-\frac{m}{2}e^{at}\ln(1-e^{-at})}$.
	
	\textbf{Proof of Claim:} By condition we have: for any fixed $\epsilon>0$, there exists $t_n\in(0,\infty)$ such that $t_n\to\infty$, $f(x,t_n)\le\frac{m}{2}+\epsilon$ ($f=\ln u$). Then the same proof as (4a) can go to (5.50) and ODE comparison gives 
	\begin{eqnarray}
	f(x,t)&\le&-\frac{m}{2}e^{at}\ln(1-e^{-at})+c(x,t_n)e^{at}\nonumber\\
	&\le&-\frac{m}{2}e^{at}\ln(1-e^{-at})+\Big(\frac{m}{2}\ln(1-e^{-at_n})+e^{-at_n}(\frac{m}{2}+\epsilon)\Big)e^{at}\nonumber\\
	&\le&-\frac{m}{2}e^{at}\ln(1-e^{-at})+\epsilon e^{a(t-t_n)}\nonumber\\
	&\le&-\frac{m}{2}e^{at}\ln(1-e^{-at})+\epsilon
	\end{eqnarray}
	in $(0,t_n)$. By letting $n\to\infty$ first and then $\epsilon\to 0^+$, we verify the claim. It is easy to see: $\forall x\in Z_2\cup Z_3$, then $\liminf\limits_{t\to\infty}\ln u(x,t)\le \frac{m}{2}$. Hence (4b) is implied by above claim. The remainder of (4b) is trivail by the definitions of $Z_1$,$Z_2$,$Z_3$. As before, if we choose solutions as (5.27), when $t\to\infty$, we have 
	\begin{equation}
	\begin{cases}
	\ln u(x,\cdot)\sim -\frac{a\parallel x-x_0\parallel^2}{4}+\frac{m}{2}+Ce^{at}\quad\text{if}\quad C>0\vspace{5mm}\nonumber\\
	\ln u(x,\cdot)\sim -\frac{a\parallel x-x_0\parallel^2}{4}+\frac{m}{2}\qquad\qquad\,\text{if}\quad C=0\vspace{5mm}\nonumber\\
	\ln u(x,\cdot)\sim -\frac{a\parallel x-x_0\parallel^2}{4}+\frac{m}{2}+Ce^{at}\quad\text{if}\quad C<0.\\
	\end{cases}
	\end{equation}
	Moreover, at case $C> 0$, then $Z_1=\mathbf{M^{n}}$; at case $C=0$, $Z_2=\mathbf{M^{n}}\setminus\{x_0\}$ and $Z_3=\{x_0\}$; at case $C<0$, $Z_2=\mathbf{M^{n}}$. In each case, our estimates in (4a) and (4b) are sharp.\\
	(5) Actually, inequality (2.12) have been proved in the process of the proof of (4) because we do not use the sign of $a$ before ODE comprison theorem, here we just need check the sharpness of (2.12). If we choose solutions as (5.27), we have 
	\begin{equation}
	F(x,t)=C+\frac{a\parallel x-x_0\parallel^2}{4(1-e^{at})},
	\end{equation}
	which is indeed increasing for $x\neq x_0$ and constant for $x=x_0$. The constant case corresponds to the equality of (2.12).\vspace{1mm}\\
	(6) For $a>0$ case, one can directly get the upper bound $e^{\frac{m}{2}}$ from (3b). One also can get it from (4a): if $u(x)=u(x,t)>e^{\frac{m}{2}}$, then $x\in Z_1$ which yields that $u$ blows up at $\infty$. This is impossible because $u$ is static on time.
	
	For $a<0$ case, one can directly get the lower bound $1$ from (5). One also can get it from (2): if $u(x)=u(x,t)<1$, then $x\in S_3$ which yields that $u$ decays to zero at $-\infty$. This is impossible because $u$ is static on time.
\end{proof}

\section{\textbf{Application to Yamabe type equation}}
In the section, we will consider global differential Harnack of Yamabe type equation {\rm(\ref{geq})}. Especially, we have proved sharp differential Harnack, Harnark inequality and Liouville type theorem under the Bakry-\'{E}mery curvature non-negative condition. We can assume $p_i\ne 1$ for some $i$, because that case is linear equation, whose differential Harnack inequalities are directly derived in Appendix 7.1. 

As before, we have
\begin{equation}
h(f):=\frac{\sum_{i=1}^{N}a_i e^{p_i f}}{e^f}=\sum_{i=1}^{N}a_i e^{(p_i-1)f}.
\end{equation}
Hence solving $A_3$-system becomes
\begin{align}
	\begin{cases}
		\frac{4\gamma}{m}c+\sum_{i=1}^{N}[a_i(p_i-1)(\alpha-\gamma)+\alpha a_i(p_i-1)^2]-2K\gamma-\gamma'\ge \frac{\gamma}{\alpha}\Big(\frac{4\gamma}{m}c-\alpha'\Big)\nonumber\\	
		\varphi'-\frac{2\gamma}{m}c^2+\frac{\varphi}{\alpha}\Big(\frac{4\gamma}{m}c-\alpha\cdot \sum_{i=1}^{N}a_i(p_i-1) e^{(p_i-1)f}-\alpha'\Big)\ge 0\nonumber\\
		\sum_{i=1}^{N}a_i(p_i-1) e^{(p_i-1)f}+\frac{\alpha'}{\alpha}+\frac{\beta'}{\beta}-\frac{4\gamma\varphi}{m\alpha^2}\le 0\nonumber\\
		\beta(0)=0\quad\text{and}\quad\beta >0\nonumber
	\end{cases}
\end{align}
on $\mathbf{M^{n}}\times(0,\infty)$ and any one of (I)-(III)
holds.

 In the following description, we assume $K=0$. 
\begin{proof}[\textbf{Proof of Theorem 2.6}]
 Under the conditions in Theorem 2.6, without loss of generality, we can assume $p_1<1$.  We just require  $a_i(p_i-1)(\alpha-\gamma)+(p_i-1)^2\ge 0$ for all $i$. Then above $A_3$-system is implied by the following $A_3$-system of heat equation:
	\begin{align}
	\begin{cases}
	\frac{4\gamma}{m}c-2K\gamma-\gamma'\ge \frac{\gamma}{\alpha}\left(\frac{4\gamma}{m}c-\alpha'\right)\nonumber\\	
	\varphi'-\frac{2\gamma}{m}c^2+\frac{\varphi}{\alpha}\Big(\frac{4\gamma}{m}c-\alpha'\Big)\ge 0\nonumber\\
\frac{\alpha'}{\alpha}+\frac{\beta'}{\beta}-\frac{4\gamma\varphi}{m\alpha^2}\le 0\nonumber\\
	\beta(0)=0\quad\text{and}\quad\beta >0.\nonumber
	\end{cases}
	\end{align}
	We choose (if $p_N\le0$, we choose $\alpha>1$)
	\begin{eqnarray}
	\gamma(t)=1, \qquad \alpha(t)=\alpha\in(1,\alpha_0],\qquad \beta(t)=t,\nonumber\\
	\varphi(t)=\frac{m\alpha^2}{2t}+\frac{m\alpha^2 K}{4(\alpha-1)},\quad c(t)=\frac{m\alpha}{2t}+\frac{m\alpha K}{2(\alpha-1)},
	\end{eqnarray}
	where
	\begin{eqnarray}
		\alpha_0:=
		\begin{cases}
	\frac{1}{p_N}\quad\qquad \text{if}\qquad 0<p_N<1\\
	\frac{1}{p_{N-1}}\,\qquad \text{if}\qquad p_N=1.\nonumber
		\end{cases}
	\end{eqnarray}
	By Theorem 2.2 (satisfies (I)), we have the global estimate:
	\begin{equation}
	\frac{|\nabla u|^2}{u^2}-\alpha
	\frac{u_t}{u}+\sum_{i=1}^{N}a_i u^{p_i-1}\leq\frac{m\alpha^{2}}{2t}+\frac{m\alpha^{2}K}{4(\alpha-1)}.
	\end{equation}
	By letting $\alpha\to 1^{+}$ and $K=0$, we get (2.15).
	
	For compact case, we just choose $\alpha=1$ and other functions are as same as (6.2). Then one can see that they solve $A_2$-system and yield the same differential Harnack inequality as (6.3) (hence \eqref{sgdhk}).	
\end{proof}
Then we immediately obtain the corresponding Harnack inequality.
\begin{proof}[\textbf{Proof of Theorem 2.7}]
	Define the function $s:[t_1,t_2]\rightarrow\mathbb{R}$ as
	\begin{equation}
	s(t)=f(l(t),t),
	\end{equation}
	where $l\in \mathscr{P}$ and $\mathscr{P}$ denotes all smooth paths connecting $x_1$ and $x_2$ ($l(t_i)=x_i,i=1,2$). By the chain rule and global differential Harnack (2.3), we get
	\begin{equation}
	s'(t)=f_t+\nabla f\cdot\frac{dl}{dt}
	\end{equation}
	and 
	\begin{eqnarray}
	\frac{d}{dt}(s(t))&=&f_t+\nabla f\cdot\frac{dl}{dt}\nonumber\\
	&\ge&\frac{\gamma |\nabla f|^2-\varphi}{\alpha}+\nabla f\cdot\frac{dl}{dt}\nonumber\\
	&\ge&\frac{-\varphi}{\alpha}-\frac{\alpha}{4\gamma}\left|\frac{dl}{dt}\right|^2.
	\end{eqnarray}
	Here the first inequality is due to the differential Harnack and $a_i\ge 0$ and $p_N\le 1$.
	Hence (2.16) holds by integrating both sides of (6.6) from $t_1$ to $t_2$. The remainder of Theorem 2.7 is trivial by the fact $\left|\frac{dl}{dt}\right|=\frac{d(x_1,x_2)}{t_1-t_2}$.	
\end{proof}

\begin{proof}[\textbf{Proof of Theorem 2.8}]
	(a) As in the proof of Theorem 2.5, we fix $t_1\in(-\infty,0)$. We then choose $T_0 > -t_1$ and define $\tilde{u}(x,t)=u(x,t-T_0)$ on $\mathbf{M^{n}}\times(-\infty,T_0)$. As a result, $\tilde{u}$ solves the equation on $\mathbf{M^{n}}\times(-\infty,T_0)$. By applying the differential Harnack inequality of the linear heat equation and selecting $t:=t_1+T_0$, we obtain:
	\begin{eqnarray}
		\Delta_{V}(\ln u)(x,t_1)&=&\Delta_{V}(\ln \tilde{u})(x,t)\nonumber\\
		&\ge&-\frac{m}{2(t_1+T_0)}.
	\end{eqnarray}
	By letting $T_0\to\infty$, we have	
	\begin{equation}
	\Delta_{V}(\ln u)(x,t_1)\ge 0,
	\end{equation}
	which yields
	\begin{equation}
	\partial_{t}(\ln u)\ge 0.
	\end{equation}
	Then we get $u(x,\cdot)$ is nondecreasing for any $x\in\mathbf{M^{n}}$.
	
	For $V$-heat equation case, if additional condition (2.19) holds, when $t\le t_0$,
	\begin{eqnarray}
		u(x,t)&\le& u(x,t_0)\nonumber\\
		&\le&e^{o(d(x,x_0))}\qquad\qquad\text{as}\quad d(x,x_0)\to\infty\nonumber\\
		&\le&e^{o(d(x,x_0)+\sqrt{-t})}\qquad\text{as}\quad d(x,x_0)\to\infty,t\to-\infty.
	\end{eqnarray}
	Then Souplet-Zhang's Liouville theorem gives the triviality of solution\footnote{Strictly speaking, we need the similar version of Liouville theorem in Souplet-Zhang \cite{SZ} under $Ric^m_V\ge 0$, one can readily verifies this by checking the proof in \cite{SZ} and using Bakry-Qian's Laplacian comprison theorem when one requires the Laplacian comprison theorem under Ricci curvature condition.}.\\  
	(b) Taking logarithmic of $t^{\frac{m}{2}}u(x,\cdot)$, we find that its differentiation on time is non-negative by Li-Yau's differential Harnack inequality (2.15). This completes the proof of (b).\vspace{1mm}\\
	(c) As before, for fixed $t_1\in(-\infty,0)$, then we choose $T_0>-t_1$ and define $\tilde{u}(x,t)=u(x,t-T_0)$ on $\mathbf{M^{n}}\times(-\infty,T_0)$. Then $\tilde{u}$ solves (2.14) on $\mathbf{M^{n}}\times(-\infty,T_0)$. By differential Harnack of case $p_N\le 1$ and choosing $t =t_1+T_0$, then the same argument as (a) gives:
	\begin{equation}
	\Delta_{V}(\ln u)(x,t_1)\ge 0,
	\end{equation}
	which yields
	\begin{equation}
		\partial_{t}u\ge \sum_{i=1}^{N}a_i u^{p_i}\ge a_1u^{p_1}.
	\end{equation}
	Without loss of generality, we assume $u$  define on $\mathbf{M^n}\times(-\infty,0]$ and there exists $x_0\in\mathbf{M^{n}}$ such that $u(x_0,0)>0$, 
	%we define $M(x_0,r):=\sup_{x\in B_{x_0}(r)}u(x,0)$. 
	then we natually consider Cauchy problem:
	\begin{equation}
		\begin{cases}
	        \dot{v}=a_1 v^{p_1}\\
	       v(0)=u(x_0,0).
		\end{cases}
	\end{equation}
	By directly solving (6.13) and applying ODE comparison principle, we have
	\begin{eqnarray}
	u(x_0,t)\le\left[a_1(1-p_1)\Big(t+\frac{u(x_0,0)^{1-p_1}}{a_1(1-p_1)}\Big)\right]^{\frac{1}{1-p_1}}.
	\end{eqnarray}
	Hence $u(x_0,t_0)=0$, where $t_0=-\frac{u(x_0,0)^{1-p_1}}{a_1(1-p_1)}$. 
	Therefore, there does not exist positive solution of {(\ref{geq})} on $\mathbf{M^{n}}\times(-\infty,0)$. Futhermore, strong maximal principle gives $u\equiv 0$ before some time $T\in(-\infty,0)$.
	%If $p_i=0$ for some $i$, then above proof from (6.12) to (6.14) goes again if one substitute $a_1,p_1$ into $a_i,p_i$, which yields that there is a zero $t$ of $u(x_0,\cdot)$ such that $u(x_0,t)\to-\infty$ as $t\to-\infty$. 
	Then we complete the proof of (c). \\  
(d)\textbf{ Claim:} If $x\in\mathbf{M^{n}}$ such that 
\begin{equation}
\lim_{t\to\infty}\frac{u^{1-p_1}(x,t)}{t} =0,
\end{equation}
then we have
\begin{equation}
\liminf_{t\to\infty}\frac{u(x,t)}{u_0(t)}\ge 1,
\end{equation}
where $u_0$ is defined as in Theorem 2.8.

\textbf{Proof of Claim:} Differential Harnack directly gives:
\begin{equation}
	\partial_{t}u\ge \left(-\frac{mu^{1-p_1}}{2t}+a_1\right) u^{p_1}.
\end{equation}
For fixed $\epsilon\in(0,a_1)$, by (6.15), there exists $t_0$ such that $-\frac{mu^{1-p_1}(x,t)}{2t}+a_1\ge \epsilon$ when $t\ge t_0$. Then ODE comparison principle directly gives 
\begin{equation}
	u(x,t)\ge u_{\epsilon}(t) \qquad\text{for}\qquad t\ge t_0,
\end{equation}
where $u_{\epsilon}$ is solution of the following Cauchy problem:
\begin{eqnarray}
	\begin{cases}
	\dot{v}=\epsilon\cdot  v^{p_1}\nonumber\\
	v(t_0)=u(x,t_0).
	\end{cases}
\end{eqnarray} 
Then we have 
\begin{equation}
	\liminf_{t\to\infty}\frac{u(x,t)}{u_0(t)}=\liminf_{t\to\infty}\frac{u(x,t)}{u_{\epsilon}(t)}\cdot\frac{u_{\epsilon}(t)}{u_{0}(t)}\ge\Big(\frac{\epsilon}{a_1}\Big)^{\frac{1}{1-p_1}}.
\end{equation}
Letting $\epsilon\to a_1^-$, we get (6.16).

It is easy to see that (6.16) controdicts with (6.15) when $p_1<1$. Hence there does not exist positive solution on $\mathbf{M^{n}}\times(0,\infty)$ such that $u(x,t)=o(t^{\frac{1}{1-p_1}})$ for some $x\in\mathbf{M^{n}}$.
\end{proof}

\section{\textbf{Appendix}}
In Appendix, we give some general solutions of $A_3$-system about some specific equations. Some solutions answer some questions asked by previous authors, some solutions are new and improve many results of other authors. 

\subsection{\textbf{Solutions of $A_3$-system related to linear equation}}
In this subsection, we briefly derive some classical estimates for linear heat equation by solving the corresponding $A_3$-system. 

We consider equation:
\begin{equation}
\partial_{t}u=\Delta_{V}u+pu,
\end{equation}
where $p$ is a real constant.
Then 
\begin{equation}
h(f):=\frac{p\cdot e^f}{e^f}=p.
\end{equation}
Hence solving $A_3$-system becomes
\begin{align}
	\begin{cases}
		\frac{4\gamma}{m}c-2K\gamma-\gamma'\ge \frac{\gamma}{\alpha}\Big(\frac{4\gamma}{m}c-\alpha'\Big)\nonumber\\	
		\varphi'-\frac{2\gamma}{m}c^2+\frac{\varphi}{\alpha}\Big(\frac{4\gamma}{m}c-\alpha'\Big)\ge 0\nonumber\\
		\frac{\alpha'}{\alpha}+\frac{\beta'}{\beta}-\frac{4\gamma\varphi}{m\alpha^2}\le 0\nonumber\\
		\beta(0)=0\quad\text{and}\quad\beta >0\nonumber
	\end{cases}
\end{align}
and any one of (I)-(III) holds.

Then we give some classical solutions of $A_3$-system:\\
(i) Li-Yau and Davies type:\\We choose

\begin{eqnarray}
\gamma(t)=1, \qquad \alpha(t)=\alpha>1,\qquad \beta(t)=t,\nonumber\\
\varphi(t)=\frac{m\alpha^2}{2t}+\frac{m\alpha^2 K}{4(\alpha-1)},\quad c(t)=\frac{m\alpha}{2t}+\frac{m\alpha K}{2(\alpha-1)}.
\end{eqnarray}
By Theorem 2.2 (satisfies (I)), we have
\begin{equation}
\frac{|\nabla u|^2}{u^2}-\alpha
\frac{u_t}{u}+\alpha p\leq\frac{m\alpha^{2}}{2t}+\frac{m\alpha^{2}K}{4(\alpha-1)}.
\end{equation}
If $V=\mathbf{0}$ and $p=0$, letting $m\to n^+$, we obtain the original Davies's estimate. \\
(ii) Li-Xu type:\\We choose

\begin{eqnarray}
\gamma(t)=1, \quad \alpha(t)=1+\frac{\sinh(Kt)\cosh(Kt)-1}{\sinh^2(Kt)},\quad \beta(t)=\tanh(Kt),\nonumber\\
\varphi(t)=\frac{mK}{2}[\coth(Kt)+1],\qquad c(t)=\frac{mK}{2}[\coth(Kt)+1].
\end{eqnarray}
By Theorem 2.2 (satisfies (II) and (III)), we have 
\begin{equation}
\frac{|\nabla u|^2}{u^2}-\alpha(t)\frac{u_t}{u}+\alpha(t)p\leq\frac{mK}{2}[\coth(Kt)+1].
\end{equation}
If $V=\mathbf{0}$ and $p=0$, letting $m\to n^+$, we obtain the original Li-Xu's estimate. \\
(iii) Linear Li-Xu type:\\We choose

\begin{eqnarray}
\gamma(t)=1, \qquad \alpha(t)=1+\frac{2}{3}Kt,\qquad \beta(t)=\tanh(Kt),\nonumber\\
\varphi(t)=\frac{m}{2}\Big(\frac{1}{t}+K+\frac{1}{3}K^2t\Big),\quad c(t)=\frac{m}{2}(\frac{1}{t}+K).
\end{eqnarray}
By Theorem 2.2 (satisfies (II) and (III)), we have 
\begin{equation}
\frac{|\nabla u|^2}{u^2}-\Big(1+\frac{2}{3}Kt\Big)
\frac{u_t}{u}+\Big(1+\frac{2}{3}Kt\Big)p\leq\frac{m}{2}(\frac{1}{t}+K).
\end{equation}
If $V=\mathbf{0}$ and $p=0$, letting $m\to n^+$, we obtain the original linear Li-Xu's estimate. \\
(iv) Hamilton type:\\We choose

\begin{eqnarray}
\gamma(t)=\delta\cdot e^{-2Kt}, \qquad \alpha(t)=1,\qquad \beta(t)=t,\nonumber\\
\varphi(t)=\frac{m\cdot e^{2Kt}}{2\delta t},\qquad c(t)=\frac{m\cdot e^{2Kt}}{2\delta t},
\end{eqnarray}
where $\delta\in(0,1)$. By Theorem 2.2 (satisfies (I)), we have
\begin{equation}
\delta\cdot e^{-2Kt}\frac{|\nabla u|^2}{u^2}-\frac{u_t}{u}+p\leq\frac{m\cdot e^{2Kt}}{2\delta t}.
\end{equation}
If $V=\mathbf{0}$ and $p=0$, letting $m\to n^+$ and $\delta\to 1^-$, we get 
\begin{equation}
e^{-2Kt}\frac{|\nabla u|^2}{u^2}-\frac{u_t}{u}\leq\frac{n\cdot e^{2Kt}}{2t}.
\end{equation}
When $\mathbf{M^n}$ is a closed manifold, (7.11) recovers the original Hamilton estimate.\\
\textbf{Remark 7.1}\\
(a) While we have presented the above estimates, which have also been derived by other authors, it is worth noting that one can derive other types of differential Harnack inequalities by solving the $A_3$-system. These estimates can then be used to obtain corresponding bounds for the heat kernel and Green function.\\
(b) We acknowledge that the Davies type estimate (7.4) and the Hamilton type estimate (complete manifold case) (7.11) were derived in \cite{D} and \cite{HM} respectively. However, our proof here differs from theirs. They both utilized Davies' technique to analyze the quadratic inequality of the auxiliary function (also see \cite{HHL,HM,M,Y1}). In contrast, we do not employ this technique in the proof of Theorem 2.2, and our proof is more straightforward.\\
(c) Strictly speaking, for the Li-Xu and linear Li-Xu type estimates, we require $K>0$ in order to solve our $A_3$-system. However, one can obtain the sharp Li-Yau estimate for the case of Ricci non-negativity by letting $K\to 0^+$ in any of the cases.

\subsection{\textbf{Solutions of $A_3$-system related to logarithmic type equation}}
In Section 5, we have given the $A_3$-system of logarithmic type equation. Now, we directly give our solutions of above system.\\
(i) Li-Yau and Davies type \\
Case1: $a\ge0$.\\
1.If $K\le\frac{3}{2}a(\alpha-1)$ and for $\alpha>1$, we choose
\begin{eqnarray}
\gamma(t)=1, \qquad \alpha(t)=\alpha>1,\qquad \beta(t)=t,\nonumber\\
\varphi(t)=\frac{m\alpha^2}{2t}+\frac{ma\alpha^2 }{2},\quad c(t)=\frac{m\alpha}{2t}+\frac{ma\alpha}{2}.
\end{eqnarray}
By Theorem 2.2 (satisfies (I)), we have
\begin{equation}
\frac{|\nabla u|^2}{u^2}-\alpha
\frac{u_t}{u}+\alpha a\ln u\le\frac{m\alpha^2}{2t}+\frac{ma\alpha^2 }{2}.
\end{equation}
2.If $K\ge\frac{3}{2}a(\alpha-1)$ and for $\alpha>1$, we choose
\begin{eqnarray}
\gamma(t)=1, \qquad \alpha(t)=\alpha>1,\qquad \beta(t)=t,\nonumber\\
\varphi(t)=\frac{m\alpha^2}{2t}+\frac{m\alpha^2 }{2}\Big(\frac{K}{2(\alpha-1)}+\frac{a}{4}\Big),\nonumber\\ c(t)=\frac{m\alpha}{2t}+m\alpha\Big(\frac{K}{2(\alpha-1)}-\frac{a}{4}\Big).
\end{eqnarray}
By Theorem 2.2 (satisfies (I)), we have
\begin{equation}
\frac{|\nabla u|^2}{u^2}-\alpha
\frac{u_t}{u}+\alpha a\ln u\le\frac{m\alpha^2}{2t}+\frac{m\alpha^2 }{2}\Big(\frac{K}{2(\alpha-1)}+\frac{a}{4}\Big).
\end{equation}
This estimate is new and is better than previous work (see  \cite{HHL,Y1}).\\ %either $K\ge\frac{3}{2}a(\alpha-1)$ or $K\le\frac{3}{2}a(\alpha-1)$.\\
Case2: $a\le 0$.\\We choose
\begin{eqnarray}
\gamma(t)=1, \qquad \alpha(t)=\alpha>1,\qquad \beta(t)=t,\nonumber\\
\varphi(t)=\frac{m\alpha^2}{2t}+\frac{m\alpha^2 }{2}\Big(\frac{K}{2(\alpha-1)}-\frac{a}{4}\Big),\nonumber\\ c(t)=\frac{m\alpha}{2t}+m\alpha\Big(\frac{K}{2(\alpha-1)}-\frac{a}{4}\Big).
\end{eqnarray}
By Theorem 2.2 (satisfies (I)), we have
\begin{equation}
\frac{|\nabla u|^2}{u^2}-\alpha
\frac{u_t}{u}+\alpha a\ln u\le\frac{m\alpha^2}{2t}+\frac{m\alpha^2 }{2}\Big(\frac{K}{2(\alpha-1)}-\frac{a}{4}\Big).
\end{equation}
Here, we should mention the process of finding these functions. By observing classical Li-Yau's estimate, it is reasonable to guess $\varphi$ and $c$ have the following form:
\begin{eqnarray}
\varphi=\frac{l_1}{t}+l_2,\qquad c(t)=\frac{l_3}{t}+l_4.\nonumber
\end{eqnarray}
Then we plug above form to the $A_3$-system and solve all possible solutions. We omit these cumbersome computations for shortening the length of paper.\\
(ii) Li-Xu type\\
Case1: $a\ge 0$.\\We choose
\begin{eqnarray}
\gamma(t)=1, \quad  \beta(t)=\tanh((K+\frac{a}{2})t),\nonumber\\
\alpha(t)=\frac{a+2K}{a+K}\cdot\frac{e^{(2K+a)t}-1+(a+2K)\frac{e^{-at}-1}{a}}{e^{(2K+a)t}+e^{-(2K+a)t}-2},\nonumber\\
c(t)=\varphi(t)=\frac{m}{2}(K+\frac{a}{2})[\coth((K+\frac{a}{2})t)+1].
\end{eqnarray}
By Theorem 2.2 (satisfies (II) and (III)), we have
\begin{eqnarray}
&&\frac{|\nabla u|^2}{u^2}-\alpha(t)\frac{u_t}{u}+ \alpha(t)a\ln u\le \varphi(t).
\end{eqnarray}
Then we recover (1) of \cite[Theorem 1.2]{HHL}. One also can select the following solution:
\begin{eqnarray}
\gamma(t)=1, \quad \alpha(t)=1+\frac{2K}{3}\frac{1-e^{-at}}{a}, \nonumber\\
\beta(t)=\tanh((K+\frac{a}{2})t),\quad c(t)=\frac{m}{2}\Big(\frac{a}{1-e^{-at}}+K+\frac{a}{2}\Big),\nonumber\\
\varphi(t)=\frac{m}{2}\Big(\frac{(K+\frac{3a}{2})^2}{2a}\frac{e^{at}+1}{e^{at}-1}-\frac{2(K+\frac{a}{2})(K+\frac{3a}{2})}{a}\frac{1}{e^{at}-1}+\frac{(K+\frac{a}{2})^2 t}{(e^{at}-1)^2}\Big).
\end{eqnarray}
This solution recovers \cite[Theorem 1.3]{HHL}. \\
Case2: $a\le 0$.\\We choose (satisfies (II) and (III)):
\begin{eqnarray}
\gamma(t)=1, \quad  \beta(t)=\tanh((K-\frac{a}{2})t),\nonumber\\
\alpha(t)=1+\frac{\sinh(K-\frac{a}{2})t\cosh(K-\frac{a}{2})t-(K-\frac{a}{2})t}{\sinh^2(K-\frac{a}{2})t},\nonumber\\
c(t)=\varphi(t)=\frac{m}{2}(K-\frac{a}{2})[\coth((K-\frac{a}{2})t)+1].
\end{eqnarray}
The solution is new and answers the question in \cite[Remark1.6]{HHL}.\\
If $-K\le a<0$, we also can choose:
\begin{eqnarray}
\gamma(t)=1, \quad  \beta(t)=\tanh((K-\frac{a}{2})t).\nonumber\\
\alpha(t)=\frac{2K-a}{K}\cdot\frac{e^{(2K-a)t}-1+(2K-a)\frac{e^{-at}-1}{a}}{e^{(2K-a)t}+e^{-(2K-a)t}-2},\nonumber\\
c(t)=\varphi(t)=\frac{m}{2}(K-\frac{a}{2})[\coth((K-\frac{a}{2})t)+1].
\end{eqnarray}
This recovers (2) of \cite[Theorem 1.2]{HHL}.\\
Because of rough computation, we give the motivation of solution (7.18) and (7.21).\\
For $a>0$ and $\gamma(t)\equiv 1$, the original $A_3$-system is implied by the following system:
\begin{align}
	\begin{cases}
		\frac{4}{m}c-2K\ge \frac{1}{\alpha}\left(\frac{4}{m}c-\alpha'\right)\nonumber\\	
		\varphi'-\frac{2}{m}c^2+\frac{\varphi}{\alpha}\left(\frac{4}{m}c-a\alpha-\alpha'\right)\ge 0\nonumber\\
		a+\frac{\alpha'}{\alpha}+\frac{\beta'}{\beta}-\frac{4\varphi}{m\alpha^2}\le 0\nonumber\\
		\beta(0)=0\quad\text{and}\quad\beta >0.\nonumber
	\end{cases}
\end{align}
Actually, we find that $\alpha$, $\varphi$, $c$ in (7.18) solve the following system (This is very similar to Li-Xu's system which drops item $-a\alpha$ in the second inequality and $a$ in third inequality, see \cite{LX}):
\begin{align}
	\begin{cases}
		\frac{4}{m}c-2K= \frac{1}{\alpha}\Big(\frac{4}{m}c-\alpha'\Big)\\	
		\varphi'-\frac{2}{m}c^2+\frac{\varphi}{\alpha}\Big(\frac{4}{m}c-a\alpha-\alpha'\Big)=0.
	\end{cases}
\end{align}
Then by a mathematical analysis argument, one finds that the third inequality also holds for functions in (7.21).

For $a<0$ and $\gamma(t)\equiv 1$, the original $A_3$-system is implied by the following system:
\begin{align}
	\begin{cases}
		\frac{4}{m}c+a-2K\ge \frac{1}{\alpha}\left(\frac{4}{m}c-\alpha'\right)\\	
		\varphi'-\frac{2}{m}c^2+\frac{\varphi}{\alpha}\left(\frac{4}{m}c-\alpha'\right)\ge 0\\
		\frac{\alpha'}{\alpha}+\frac{\beta'}{\beta}-\frac{4\varphi}{m\alpha^2}\le 0\\
		\beta(0)=0\quad\text{and}\quad\beta >0\\
		1\le \alpha(t)\le 2.
	\end{cases}
\end{align}
As before, we observe the following system which organized by the first and the second equations in $A_3$-system:
\begin{align}
	\begin{cases}
		\frac{4}{m}c-2(K-\frac{a}{2})= \frac{1}{\alpha}\left(\frac{4}{m}c-\alpha'\right)\\	
		\varphi'-\frac{2}{m}c^2+\frac{\varphi}{\alpha}\left(\frac{4}{m}c-\alpha'\right)=0.
	\end{cases}
\end{align}
Actually, this is Li-Xu's system (See \cite{LX}) with $K-\frac{a}{2}$ substituting $K$, hence direct solutions as heat equation case are given in (7.21). The third inequality holds by the same argument in \cite{LX}. Same reasons derive (7.20) and (7.22). We omit these computations for shortening the length of paper.\\
(iii) Linear Li-Xu type\\
Case1: $a\ge0$.\\
We choose (satisfies (II) and (III)):
\begin{eqnarray}
\gamma(t)=1, \qquad \alpha(t)=1+\frac{2}{3}Kt,\qquad \beta(t)=\tanh(Kt),\nonumber\\
\varphi(t)=\frac{m}{2}\Big(\frac{1}{t}+K+\frac{1}{3}K^2t\Big)+\frac{ma}{16}(at+6)(1+\frac{2Kt}{3})^2,\nonumber\\ c(t)=\frac{m}{2}(\frac{1}{t}+K)+\frac{ma}{4}(1+\frac{2Kt}{3}).
\end{eqnarray}
Case2: $a\le 0$.\\
We choose (satisfies (II) and (III)):
\begin{eqnarray}
\gamma(t)=1, \qquad \alpha(t)=1+\frac{2}{3}Kt,\qquad \beta(t)=\tanh(Kt),\nonumber\\
\varphi(t)=\frac{m}{2}\Big(\frac{1}{t}+K+\frac{1}{3}K^2t\Big)-\frac{ma}{16}(1+\frac{2Kt}{3})^2,\nonumber\\ c(t)=\frac{m}{2}(\frac{1}{t}+K)-\frac{ma}{4}(1+\frac{2Kt}{3}).
\end{eqnarray}
These estimates are new here. \\
(iv)Hamilton type\\
Case1:$a\ge 0$.\\We choose
\begin{eqnarray}
\alpha(t)=1,\qquad\gamma(t)=\delta\cdot e^{-2Kt},\qquad  \beta(t)=t,\nonumber\\
\varphi(t)=\frac{m\cdot e^{2Kt}}{2\delta }(\frac{1}{t}+a),\qquad c(t)=\frac{m\cdot e^{2Kt}}{2\delta }(\frac{1}{t}+a),
\end{eqnarray}
where $\delta\in(0,1)$. By Theorem 2.2 (satisfies (I)), we have 
\begin{equation}
\delta\cdot e^{-2Kt}\frac{|\nabla u|^2}{u^2}-\frac{u_t}{u}+a\ln u\leq\frac{m\cdot e^{2Kt}}{2\delta }(\frac{1}{t}+a).
\end{equation}
 Letting $\delta\to 1^-$, we get 
\begin{equation}
e^{-2Kt}\frac{|\nabla u|^2}{u^2}-\frac{u_t}{u}+a\ln u\le\frac{m\cdot e^{2Kt}}{2 }(\frac{1}{t}+a).
\end{equation}
If $V=-\nabla f$, under the assumption $Ric^m_f\ge-K$, this recovers \cite[Corollary 1.3]{HM}.\\
Case2:$a\le 0$.\\ We choose
\begin{eqnarray}
\alpha(t)=1,\qquad\gamma(t)=\delta\cdot e^{-2Kt},\quad  \beta(t)=t,\nonumber\\
\varphi(t)=\frac{m\cdot e^{2Kt}}{2\delta }(\frac{1}{t}-\frac{a}{8}),\qquad c(t)=\frac{m\cdot e^{2Kt}}{2\delta }(\frac{1}{t}-\frac{a}{2}),
\end{eqnarray}
where $\delta\in(0,1)$. By Theorem 2.2 (satisfies (I)), we have
\begin{equation}
\delta\cdot e^{-2Kt}\frac{|\nabla u|^2}{u^2}-\frac{u_t}{u}+a\ln u\leq\frac{m\cdot e^{2Kt}}{2\delta }(\frac{1}{t}-\frac{a}{8}).
\end{equation}
Letting $\delta\to 1^-$, we get 
\begin{equation}
e^{-2Kt}\frac{|\nabla u|^2}{u^2}-\frac{u_t}{u}+a\ln u\le\frac{m\cdot e^{2Kt}}{2 }(\frac{1}{t}-\frac{a}{8}).
\end{equation}
If $V=-\nabla f$, under the assumption $Ric^m_f\ge-K$, this improves the result of \cite[Corollary 1.3]{HM} whose right hand of (7.38) is $\frac{m\cdot e^{2Kt}}{2 }(\frac{1}{t}-\frac{a}{2})$.\\
(v)An extra type\\
Case1: $a\ge 0$.\\We choose
\begin{eqnarray}
\gamma(t)=1, \quad  \beta(t)=\tanh(Kt),\quad
\alpha(t)=1+\frac{\sinh(Kt)\cosh(Kt)-Kt}{\sinh^2 (Kt)},\nonumber\\
\varphi(t)=\frac{m}{2}(K+a)e^{at}[\coth(Kt)+1],\nonumber\\ c(t)=\frac{m}{2}\sqrt{(K+a)K}\cdot e^{at}[\coth(Kt)+1].
\end{eqnarray}
Case2: $a\le 0$.\\We choose
\begin{eqnarray}
\gamma(t)=1, \quad  \beta(t)=\tanh(Kt),\quad
\alpha(t)=1+\frac{\sinh(Kt)\cosh(Kt)-Kt}{\sinh^2 (Kt)},\nonumber\\
\varphi(t)=\frac{mK}{2}[\coth(Kt)+1]-\frac{ma}{16}\Big(1+\frac{\sinh(Kt)\cosh(Kt)-Kt}{\sinh^2 (Kt)}\Big)^2,\nonumber\\ c(t)=\frac{mK}{2}[\coth(Kt)+1]-\frac{ma}{4}\Big(1+\frac{\sinh(Kt)\cosh(Kt)-Kt}{\sinh^2 (Kt)}\Big).
\end{eqnarray}
By Theorem 2.2 (satisfies (II) and (III)), we  obtain the corresponding differential Harnack inequality. This is a new estimate and we omit its computations.\\

\subsection{\textbf{Solutions of $A_3$-system related to Yamabe type equation}} In Section 6, we have given $A_3$-system of Yamabe type equation. Now, we directly give our solutions of above system. In many cases, we require the upper bound of solution to derive  solutions of $A_3$-system.
Define:
\begin{equation}
M=\sup_{y\in\mathbf{M^{n}}}|b(p-1)|u^{p-1}(y).
\end{equation}
In the following estimates, if $M$ occurs in the formulas, it means that $M$ is finite.\\
\textbf{Case1}: $b(p-1)> 0$.\\
Case 1.1: $b>0, p>1$.\\
(i) Li-Yau type \\
We choose (satisfies (I)):
\begin{eqnarray}
\gamma(t)=1, \qquad \alpha(t)=\alpha>1,\qquad \beta(t)=t,\nonumber\\
\varphi(t)=\frac{m\alpha^2}{2t}+\frac{mM\alpha^2 }{2}+\frac{\alpha^2 mK}{2(\alpha-1)},\nonumber\\
c(t)=\frac{m\alpha}{2t}+\frac{mM\alpha}{2}+\frac{\alpha mK}{2(\alpha-1)}.
\end{eqnarray}
(ii) Li-Xu type
\\We choose (satisfies (II) and (III)):
\begin{eqnarray}
\gamma(t)=1, \quad  \beta(t)=\tanh((K+\frac{M}{2})t),\nonumber\\
\alpha(t)=\frac{M+2K}{M+K}\cdot\frac{e^{(2K+M)t}-1+(M+2K)\frac{e^{-Mt}-1}{M}}{e^{(2K+M)t}+e^{-(2K+M)t}-2},\nonumber\\
c(t)=\varphi(t)=\frac{m}{2}(K+\frac{M}{2})[\coth((K+\frac{M}{2})t)+1].
\end{eqnarray}
(iii)  Linear Li-Xu type\\
We choose (satisfies (II) and (III)):
\begin{eqnarray}
\gamma(t)=1, \qquad \alpha(t)=1+\frac{2}{3}Kt,\qquad \beta(t)=\tanh(Kt),\nonumber\\
\varphi(t)=\frac{m}{2}\Big(\frac{1}{t}+K+\frac{1}{3}K^2t\Big)+\frac{mM}{16}(Mt+6)(1+\frac{2Kt}{3})^2,\nonumber\\ c(t)=\frac{m}{2}(\frac{1}{t}+K)+\frac{mM}{4}(1+\frac{2Kt}{3}).
\end{eqnarray}
(iv)Hamilton type\\
We choose (satisfies (I)):
\begin{eqnarray}
\alpha(t)=1,\qquad\gamma(t)=\delta\cdot e^{-2Kt},\quad  \beta(t)=t,\nonumber\\
\varphi(t)=\frac{m\cdot e^{2Kt}}{2\delta }(\frac{1}{t}+M),\qquad c(t)=\frac{m\cdot e^{2Kt}}{2\delta }(\frac{1}{t}+M),
\end{eqnarray}
where $\delta\in(0,1)$. By Theorem 2.2 we have 
\begin{equation}
\delta\cdot e^{-2Kt}\frac{|\nabla u|^2}{u^2}-\frac{u_t}{u}+a+bu^{p-1}\leq\frac{m\cdot e^{2Kt}}{2\delta }(\frac{1}{t}+M).
\end{equation}
By letting $\delta\to 1^-$, we have
\begin{equation}
e^{-2Kt}\frac{|\nabla u|^2}{u^2}-\frac{u_t}{u}+a+bu^{p-1}\leq\frac{m\cdot e^{2Kt}}{2}(\frac{1}{t}+M).
\end{equation}
Case 1.2: $b<0, p\in(0,1)$.\\
In this case, for Li-Yau, Li-Xu and linear Li-Xu type estimates, we require $\gamma(t)=1$. If we have $\alpha\ge\frac{1}{p}$, $A_3$-system here is implied by the $A_3$-system of heat equation.
With the following simple observation, we will directly give solutions of $A_3$-system.
If $\alpha$, $\beta$, $\varphi$, $c$ solve the system
\begin{align}
\begin{cases}
\frac{4}{m}c-2K\ge \frac{1}{\alpha}\Big(\frac{4}{m}c-\alpha'\Big)\nonumber\\	
\varphi'-\frac{2}{m}c^2+\frac{\varphi}{\alpha}\Big(\frac{4}{m}c-\alpha M-\alpha'\Big)\ge 0\nonumber\\
M+\frac{\alpha'}{\alpha}+\frac{\beta'}{\beta}-\frac{4\varphi}{m\alpha^2}\le 0\nonumber\\
\beta(0)=0\quad\text{and}\quad\beta >0,\nonumber
\end{cases}
\end{align}
for any $k\ge 1$, $k\alpha$, $\beta$, $k^2\varphi$, $kc$ also solve the same system.

By Case 1.1, we have got a solution of above system, so if we set 
\begin{eqnarray}
\tilde{\alpha}=k\alpha,\qquad\tilde{\beta}=\beta,\nonumber\\
\tilde{\varphi}=k^2\varphi, \qquad \tilde{c}=k c,\nonumber
\end{eqnarray}
for $k\ge\frac{1}{p}$, where $\alpha$, $\beta$, $\varphi$, $c$ are functions in Case 1.1, which satisfy Li-Yau, Li-Xu and linear Li-Xu estimates, respectively. Then $\tilde{\alpha}$, $\tilde{\beta}$, $\tilde{\varphi}$, $\tilde{c}$ solve the origin $A_3$-system and consequently get Li-Yau, Li-Xu and linear Li-Xu type estimates in this case.

For Hamilton type estimate, we first choose $\alpha(t)\equiv 1$ and $\gamma(t)=\delta e^{-2Kt}$. 
Then we have two choices as follows:\\
(i) If $\delta\le p$, the same functions as (7.40) solve above system. And we get
\begin{equation}
\delta\cdot e^{-2Kt}\frac{|\nabla u|^2}{u^2}-\frac{u_t}{u}+a+bu^{p-1}\leq\frac{m\cdot e^{2Kt}}{2\delta }(\frac{1}{t}+M).
\end{equation}
(ii)
If $\delta\in(0,\frac{2}{3}]$, then we also choose the same functions as (7.40) and derive the same estimate as (7.43).
If $\delta\in(\frac{2}{3},1)$, we set
\begin{equation}
k=\frac{\delta}{2(1-\delta)}>1,
\end{equation}
then we choose
\begin{eqnarray}
\alpha(t)=1,\qquad\gamma(t)=\delta\cdot e^{-2Kt},\qquad  \beta(t)=t,\nonumber\\
\varphi(t)=\frac{k^2 m\cdot e^{2Kt}}{2\delta }(\frac{1}{t}+M),\qquad c(t)=\frac{km\cdot e^{2Kt}}{2\delta }(\frac{1}{t}+M).
\end{eqnarray}
Direct computation confirms that (7.45) solves the $A_3$-system indeed. \\
Case 1.3: $b<0, p\in(-\infty,0)$.\\ 
In this case, for Li-Yau, Li-Xu and linear Li-Xu type estimates, we require $\gamma(t)=1$. Then the original $A_3$-system is implied by the following system
\begin{align}
\begin{cases}
\frac{4}{m}c-[2|p|+1]M-2K\ge \frac{1}{\alpha}\Big(\frac{4}{m}c-\alpha'\Big)\nonumber\\	
\varphi'-\frac{2}{m}c^2+\frac{\varphi}{\alpha}\Big(\frac{4}{m}c-\alpha M-\alpha'\Big)\ge 0\nonumber\\
M+\frac{\alpha'}{\alpha}+\frac{\beta'}{\beta}-\frac{4\varphi}{m\alpha^2}\le 0\nonumber\\
\beta(0)=0\quad\text{and}\quad\beta >0\nonumber\\
1\le \alpha(t)\le 2.\nonumber
\end{cases}
\end{align}
If we set
\begin{equation}
\tilde{K}=K+(|p|+\frac{1}{2})M,
\end{equation}
then we get:\\
(i) Li-Yau type \\
As before, we choose (satisfies (I)):
\begin{eqnarray}
\qquad \alpha(t)=\alpha\in(1,2),\qquad \beta(t)=t,\nonumber\\
\varphi(t)=\frac{m\alpha^2}{2t}+\frac{mM\alpha^2 }{2}+\frac{\alpha^2 m\tilde{K}}{2(\alpha-1)},\nonumber\\
c(t)=\frac{m\alpha}{2t}+\frac{mM\alpha}{2}+\frac{\alpha m\tilde{K}}{2(\alpha-1)}.
\end{eqnarray}
(ii)Li-Xu type
\\We choose (satisfies (II)):
\begin{eqnarray}
\gamma(t)=1, \quad  \beta(t)=\tanh((\tilde{K}+\frac{a}{2})t),\nonumber\\
\alpha(t)=\frac{a+2\tilde{K}}{a+\tilde{K}}\cdot\frac{e^{(2\tilde{K}+a)t}-1+(a+2\tilde{K})\frac{e^{-at}-1}{a}}{e^{(2\tilde{K}+a)t}+e^{-(2\tilde{K}+a)t}-2},\nonumber\\
c(t)=\varphi(t)=\frac{m}{2}(\tilde{K}+\frac{a}{2})[\coth((\tilde{K}+\frac{a}{2})t)+1].
\end{eqnarray}
(iii)Linear Li-Xu type\\
For this case, we cannot restrict the upper bound of $\alpha(t)$ and must solve original $A_3$-system, direct computation  gives a solution (satisfies (I)):
\begin{eqnarray}
\qquad \alpha(t)=(2-p)(1+\frac{2}{3}Kt),\qquad \beta(t)=\tanh(Kt),\nonumber\\
\gamma(t)=1,\quad c(t)=\frac{m(2-p)}{2}(\frac{1}{t}+K)+\frac{(2-p)mM}{4}(1+\frac{2Kt}{3}),\nonumber\\
\varphi(t)=\frac{m(2-p)^2}{2}\Big(\frac{1}{t}+K+\frac{1}{3}K^2t\Big)+\frac{(2-p)^2mM}{16}(Mt+6)(1+\frac{2Kt}{3})^2.
\end{eqnarray}
(iv)Hamilton type\\
(i) If $p\le -2$, we define $k:=\frac{\delta-p}{2(1-\delta)}\ge 1$.
We choose (satisfies (I)):
\begin{eqnarray}
\alpha(t)=1,\quad\gamma(t)=\delta\cdot e^{-2Kt},\quad  \beta(t)=t,\nonumber\\
\varphi(t)=\frac{k^2m\cdot e^{2Kt}}{2\delta }(\frac{1}{t}+M),\qquad c(t)=\frac{km\cdot e^{2Kt}}{2\delta }(\frac{1}{t}+M),
\end{eqnarray}
where $\delta\in(0,1)$. \\
(ii) If $p\in (-2,0)$, we choose (satisfies (I)):
\begin{eqnarray}
\alpha(t)=1,\quad\gamma(t)=\delta\cdot e^{-2Kt},\quad  \beta(t)=t,\nonumber\\
\varphi(t)=\frac{m\cdot e^{2Kt}}{2\delta }(\frac{1}{t}+M),\qquad c(t)=\frac{m\cdot e^{2Kt}}{2\delta }(\frac{1}{t}+M).
\end{eqnarray}
If $\delta\in(\frac{2+p}{3},1)$, we choose the same functions as in (i) and the same estimate holds.\\
\textbf{Case2}: $b(p-1)< 0$.\\
Case 2.1: $b>0$, $p\in(-\infty,0)$.\\
We notice that the $A_3$-system corresponding to heat equation implies the $A_3$-system of Yamabe type equation. Then all type estimates of heat equation hold for this case. Concretely, we have
\begin{equation}
\gamma(t)\frac{|\nabla u|^2}{u^2}-\alpha(t)\frac{u_t}{u}+\alpha(t)qu^{p-1}\le \varphi(t),
\end{equation}
where $\alpha$, $\gamma$, $\varphi$ are the same as in Subsection 7.1.\\
Case 2.2: $b>0$, $p\in(0,1)$.\\
(i) Li-Yau type\\
 We choose (satisfies (I)):
\begin{eqnarray}
\gamma(t)=1, \qquad \alpha(t)=\alpha\in(1,\frac{1}{p}],\qquad \beta(t)=t,\nonumber\\
\varphi(t)=\frac{m\alpha^2}{2t}+\frac{m\alpha^2 K}{4(\alpha-1)},\quad c(t)=\frac{m\alpha}{2t}+\frac{m\alpha K}{2(\alpha-1)}.
\end{eqnarray}
(ii) Li-Xu type \\
Case1: $p\in(0,\frac{1}{2}]$, then original functions of Li-Xu type of heat equation satisfy.\\
Case2: $p\in(\frac{1}{2},1)$. Let $t(p)\in(0,\infty)$ be the unique zero of $\alpha(t)-\frac{1}{p}$, here $\alpha(t)=1+\frac{\sinh(t)\cosh(t)-t}{\sinh^2(t)}$. We choose (satisfies (II) and (III)):
\begin{eqnarray}
\gamma(t)=1, \quad \alpha(t)=1+\frac{\sinh(Kt)\cosh(Kt)-1}{\sinh^2(Kt)},\quad \beta(t)=\tanh(Kt),\nonumber\\
\varphi(t)=\frac{mK}{2}[\coth(Kt)+1],\qquad c(t)=\frac{mK}{2}[\coth(Kt)+1].
\end{eqnarray}
Then we have the same formula in interval $(0,\frac{t(p)}{K}]$ as case 1.\\
(iii) Linear Li-Xu type\\
We choose (satisfies (II) and (III)):
\begin{eqnarray}
\alpha(t)=1+\frac{2}{3}Kt,\qquad \beta(t)=\tanh(Kt),\nonumber\\
\varphi(t)=\frac{m}{2}\Big(\frac{1}{t}+K+\frac{1}{3}K^2t\Big),\quad c(t)=\frac{m}{2}(\frac{1}{t}+K).
\end{eqnarray}
When $t\in(0,\frac{3}{2K}(\frac{1}{p}-1)]$, then the same formula as heat equation case holds.\\
(iv) Hamilton type\\
We choose
\begin{eqnarray}
\gamma(t)=\delta\cdot e^{-2Kt}, \qquad \alpha(t)=1,\qquad \beta(t)=t,\nonumber\\
\varphi(t)=\frac{m\cdot e^{2Kt}}{2\delta t},\qquad c(t)=\frac{m\cdot e^{2Kt}}{2\delta t},
\end{eqnarray}
where $\delta\in(p,1)$. When $t\in(0,\frac{1}{2K}\ln(\frac{\delta}{p})]$, then the same formula as heat equation case holds.\\
Case 2.3: $b<0$, $p>1$.\\
In this case, one can compare with Case 1.3. For Li-Yau and Li-Xu type estimates, we require $\gamma(t)=1$. Then the original $A_3$-system is implied by the following system
\begin{align}
\begin{cases}
\frac{4}{m}c-[\alpha p-1]M-2K\ge \frac{1}{\alpha}\Big(\frac{4}{m}c-\alpha'\Big)\\	
\varphi'-\frac{2}{m}c^2+\frac{\varphi}{\alpha}\Big(\frac{4}{m}c-\alpha'\Big)\ge 0\\
\frac{\alpha'}{\alpha}+\frac{\beta'}{\beta}-\frac{4\varphi}{m\alpha^2}\le 0\\
\beta(0)=0\quad\text{and}\quad\beta >0\\
1\le \alpha(t)\le 2.
\end{cases}
\end{align}
If we set
\begin{equation}
\overline{K}=K+(p-\frac{1}{2})M,
\end{equation}
then we get:\\
(i) Li-Yau type \\
We choose (satisfies (I)):
\begin{eqnarray}
\qquad \alpha(t)=\alpha\in(1,2),\qquad \beta(t)=t,\nonumber\\
\varphi(t)=\frac{m\alpha^2}{2t}+\frac{\alpha^2 m\overline{K}}{4(\alpha-1)},\quad
c(t)=\frac{m\alpha}{2t}+\frac{\alpha m\overline{K}}{2(\alpha-1)}.
\end{eqnarray}
(ii)Li-Xu type\\
We choose (satisfies (II), (III) and $1\le\alpha(t)\le 2$):
\begin{eqnarray}
\alpha(t)=1+\frac{\sinh(\overline{K}t)\cosh(\overline{K}t)-\overline{K}t}{\sinh^2(\overline{K}t)},\quad\beta(t)=\tanh(\overline{K}t),\nonumber\\
\varphi(t)=c(t)=\frac{m\overline{K}}{2}[\coth(\overline{K}t)+1].
\end{eqnarray}
(iii) Linear Li-Xu type\\
Direct computation gives a solution (satisfies (II) and (III)):
\begin{eqnarray}
\qquad \alpha(t)=(2-p)(1+\frac{2}{3}Kt),\qquad \beta(t)=\tanh(Kt),\nonumber\\
\varphi(t)=\frac{m(2-p)^2}{2}\Big(\frac{1}{t}+K+\frac{1}{3}K^2t\Big)+\frac{(2-p)^2mM}{16}(Mt+6)(1+\frac{2Kt}{3})^2,\nonumber\\ c(t)=\frac{m(2-p)}{2}(\frac{1}{t}+K)+\frac{(2-p)mM}{4}(1+\frac{2Kt}{3}).
\end{eqnarray}
(iv) Hamilton type\\
(1) If $p\ge 2$, we define $k=\frac{p-\delta}{2(1-\delta)}\ge 1$.
We choose (satisfies (I)):
\begin{eqnarray}
\alpha(t)=1,\quad\gamma(t)=\delta\cdot e^{-2Kt},\quad  \beta(t)=t,\nonumber\\
\varphi(t)=\frac{k^2m\cdot e^{2Kt}}{2\delta }(\frac{1}{t}+M),\qquad c(t)=\frac{km\cdot e^{2Kt}}{2\delta }(\frac{1}{t}+M),
\end{eqnarray}
where $\delta\in(0,1)$.\\
(2) If $p\in(0,2)$ and $\delta\in(0,2-p]$, we choose (satisfies (I)):
\begin{eqnarray}
\alpha(t)=1,\quad\gamma(t)=\delta\cdot e^{-2Kt},\quad  \beta(t)=t,\nonumber\\
\varphi(t)=\frac{m\cdot e^{2Kt}}{2\delta }(\frac{1}{t}+M),\qquad c(t)=\frac{m\cdot e^{2Kt}}{2\delta }(\frac{1}{t}+M).
\end{eqnarray}
If $p\in(0,2)$ and $\delta\in(2-p,1)$, we choose the same functions as in (1) and the same estimate holds.\\
\textbf{Remark 7.2}\\
For $a=0,b>0$ and $p>0$ case, J.Y.Li \cite{JL} got Li-Yau estimate even $b$ is a function. At case $p\in(0,1)$, our $\alpha$ can be chosen in the interval $(1,\frac{1}{p}]$, but they must need $\alpha=\frac{1}{p}$ by their subtle computation method. This difference yields that we can derive \eqref{sgdhk} in Ricci non-negative case which they cannot get. For suplinear case, \cite{JL} used their gradient estimate to drive the Liouville theorem of elliptic equation of (1.8) at case $1<p<\frac{n}{n-2}$ and $n\ge 4$ which did not recover Gidas-Spruck's Liouville theorem (see \cite{GS}) whose $p\in(1,\frac{n+2}{n-2})$ and $n\ge 3$. It is natural to ask whether one can recover (even improve) the main result in \cite{GS} (include Liouville theorem and their singularity decay estimates\footnote{Concretely, one should obtain the Liouville theorem for (1.8) and its generalization with $p\in (-\infty,p(n))$, where $p(n):=\infty$ for $n\in[1,2]$; $p(n):=\frac{n+2}{n-2}$ for $n>2$. It's notable that $n$ maybe a real number if one consider these estimates under Bakry-\'{E}mery curvature or RCD metric measure space. Of course, one should get universal bound estimates in any domain which cover the special case of singularity decay estimates in a neighborhood of a singularity point in \cite{GS}.}) by establishing logarithmic gradient estimate. Recently, the author of this paper confirms the question affirmatively in another paper.

\section*{Acknowledgments}This work was completed in November, 2021. The  author would like to thank Professor Jiayu Li for his support and  encouragement. This work was partially supported by NSFC [Grant Number 11721101].
Half a year of completing this work, the author finds that most results of present artical can be generalized to equations on RCD type metric measure spaces and collects these results in another note.

\end{document}